\documentclass[10pt]{extarticle}
\usepackage{mathrsfs}
\renewcommand{\mathcal}{\mathscr}
\usepackage{amsmath,amsthm,amsopn,amstext,amscd,amsfonts,amssymb}
\usepackage{mathtools}
\usepackage{tikz-cd}
\usepackage{graphicx, color, enumerate}
\usepackage[active]{srcltx}
\usepackage{pgf}
\usepackage{verbatim}
\usepackage{pgfkeys}
\usepackage{fp}
\usepackage[shortlabels]{enumitem}
\numberwithin{equation}{section}

\numberwithin{bbb}{section}
\usepackage[T1]{fontenc}
\usepackage{babel}
\usepackage{calc}
\usepackage{everypage}
\newlength\Hoffset \Hoffset -4cm
\newlength\Voffset \Voffset -1cm
\AddEverypageHook{%
    \smash{\hspace{\oddsidemargin}\hspace{\textwidth}\hspace{\Hoffset}%
    \raisebox{\Voffset-\headsep-\headheight-\topmargin}%
    {%
    }}}
\usepackage[vmargin=3cm,hmargin=2cm]{geometry}

\usepackage{fontspec}

\newcommand{\beq}{\begin{eqnarray}}
\newcommand{\eeq}{\end{eqnarray}}
\newcommand{\bq}{\begin{equation}}
\newcommand{\eq}{\end{equation}}
\newcommand{\beqn}{\begin{eqnarray*}}
\newcommand{\eeqn}{\end{eqnarray*}}

\newcommand{\vertiii}[1]{{\vert\kern-0.25ex\vert\kern-0.25ex\vert #1
    \vert\kern-0.25ex\vert\kern-0.25ex\vert}}

\newcommand{\ignore}[1]{}

\newtheorem{proposition}{Proposition}[section]
\newtheorem{theorem}{Theorem}[section]
\newtheorem{remark}{Remark}[section]
\newtheorem{lemma}{Lemma}[section]
\newtheorem{corollary}{Corollary}[section]
\newtheorem{example}[theorem]{Example}

\newtheorem{result}{Result}[section]

\numberwithin{equation}{section}

\begin{document}
\begin{center}
{\LARGE\textbf{On the Dominions of Certain Semigroups of Transformations}}\\
\vspace{0.5cm}
\large {Halima H. Assiri and Jehan A. Al-bar}
\end{center}

\vspace{0.5cm}
\hrule
\vspace{0.6cm}
\textbf{Abstract:} In the full transformation semigroup $T_n$ on a finite chain $X_n$, let $D_n=\{\alpha \in T_n:(\forall x \in X_n) \ x\alpha \le x\}$ be the subsemigroup of all order-decreasing maps of $T_n$, and let $O_n=\{\alpha \in T_n:(\forall x ,y\in X_n) \ x \le y \Rightarrow x\alpha \leq y\alpha\}$ be the subsemigroup of all order-preserving maps of $T_n$. The Catalan monoid $C_n$ is a semigroup of all order-decreasing and order-preserving full transformations of $X_n$. In this paper, it is shown that $O_n$ is closed in $T_n$. Also, the dominion of $D_n$ and the dominion of $C_n$ in $T_n$, denoted by $Dom_{T_n}(D_n)$ and $Dom_{T_n}(C_n)$, are characterized, and it is shown that they are regular idempotent-generated subsemigroups of $T_n$. Moreover, a formula for the number of their elements and their idempotents is given.

\vspace{0.3cm}
\hrule
\vspace{0.6cm}

\vspace{0.4cm}
\hrule
\vspace{0.5cm}

\section{Introduction}
\ \ \ \ \ In \cite{7}, Nasir and Umar gave a method to describe the dominion of subsemigroups of partial transformation semigroup. In this work, we will apply the same method of Proposition 1 in \cite{7} with some modifications to describe the dominion of $D_n, O_n$, and $C_n$ in $T_n$, which are subsemigroups of full transformation semigroup, but first we will go through some essential definitions that are needed.

A semigroup $S$ is called an oversemigroup of a semigroup $U$ if the later is a subsemigroup of the former. Given an oversemigroup $S$ of a semigroup $U$, an element $d\in S$ is said to be dominated by $U$ if for every semigroup $T$ and for all homomorphisms $f,g : S \rightarrow T,$ we have
 $$f|_U=g|_U \, \Longrightarrow (d)f=(d)g.$$   
The set of all elements of $S$ dominated by $U$ is called the dominion of $U$ in $S$, denoted by $Dom_{S}(U)$, and it is a subsemigroup of $S$ containing $U$. We call $U$ closed in $S$ if $Dom_{S}(U)=U$. A semigroup $U$ is called absolutely closed if $U$ is closed in every oversemigroup $S$ of $U$, and $U$ is called saturated if $Dom_{S}(U)\neq S.$ \cite{RI}.

Let $X_n=\{1,2,...,n\}$ be a finite chain. A (partial) map $\alpha:dom(\alpha)\subseteq X_n\rightarrow Im(\alpha)\subseteq X_n$ is said to be full map if $dom(\alpha)=X_n$. The full transformation semigroup $T_n$ is the (regular) semigroup of all full maps of $X_n$, and it contains $n^n$ elements. Combinatorial properties of $T_n$ have been studied over a long period and many interesting results have emerged (see \cite{6,14,12,16,1}). To identify the structure and algebraic properties of a semigroup and understand the relationship between this semigroup and its extensions into other algebraic structures, Howie and Isbell in \cite{13} studied the dominions of semigroups, such as inverse semigroups and left-simple semigroups. Scheiblich and Kayran in \cite{30} showed that $T_n$ is absolutely closed. For any transformation $\alpha$ in $T_n$ we denote
$$\alpha=
\begin{pmatrix}  A_1&A_2&...&A_r\\a_1&a_2&...&a_r
\end{pmatrix},$$
where $A_1,A_2,...,A_r$ called the blocks of $\alpha$, and $A_i=a_i\alpha^{-1}$$(i=1,2,...,r)$. The map $\alpha$ is idempotent if every block of $\alpha$ is stationary, that is, $a_i\in A_i$ for all $i$ \cite{16}. The number of idempotents of $T_n$ have been investigated by Harris and  L.Schoenfeld in \cite{10,11}.

Let $S$ be a semigroup, for the definition of the Green's relations: $\mathcal{L},\mathcal{R}, \mathcal{H}, \mathcal{D}$ and $\mathcal{J}$ on $S$, see \cite{14} or \cite{76}. A relation $\mathcal{L^*}$ defined as $(\forall a,b \in S)$, $a\mathcal{L^*}b$ if and only if $a,b$ are related by $\mathcal{L}$ in some oversemigroup of $S$. The relation $\mathcal{R^*}$ is defined dually, the two-sided version of $\mathcal{L^*}$ and $\mathcal{R^*}$ defined by $\mathcal{J^*}$. The join of the relations $\mathcal{L^*}$ and $\mathcal{R^*}$ is denoted by $\mathcal{D^*}$ and their intersection by $\mathcal{H^*}$. We generally have $\mathcal{L}\subseteq \mathcal{L^*}$, $\mathcal{R}\subseteq \mathcal{R^*}$, $\mathcal{H}\subseteq \mathcal{H^*}$, $\mathcal{D}\subseteq \mathcal{D^*}$ and $\mathcal{J}\subseteq \mathcal{J^*}$. A semigroup is said to be abundant if each $\mathcal{L^*}$-class and each $\mathcal{R^*}$-class contains an idempotents. Of course, regular semigroups are abundant and in this case $\mathcal{K^*}=\mathcal{K}$ for $\mathcal{K}$ any of $\mathcal{L}, \mathcal{R}, \mathcal{H}, \mathcal{D}$ or $\mathcal{J}$ \cite{40}. The Green's relations in the full transformation semigroup $T_n$ were characterized by Howie in \cite{6} and Clifford in \cite{1}.

The symmetric group 
$S_n=\{\alpha \in T_n : |Im(\alpha)|=n\}$ is a subsemigroup of $T_n$ consisting of all permutations on $X_n$. The singular subsemigroup $$T_n\setminus S_n=\{\alpha \in T_n: |Im(\alpha)|\leq n-1\}$$ is the set of all non-bijective transformations on $X_n$. The study of $T_n\setminus S_n$ was initiated in 1966 by Howie, who showed that this semigroup is a regular idempotent-generated subsemigroup of $T_n$ \cite{2}.
Let $$J_r=\{\alpha \in T_n : |Im(\alpha)|=r\}$$ are the $\mathcal{J}$-classes of $T_n$, then we can regard $T_n$ as partitioned into 'layers', $$J_1,J_2,...,J_{n-1},J_n.$$ The set $J_n$ is the symmetric group $S_n$, and $T_n\setminus S_n=J_1\cup J_2\cup...\cup J_{n-1}$ \cite{6}.

The partial transformation semigroup $P_n$ is the semigroup of all partial maps of $X_n$. The symmetric inverse semigroup $I_n$ is the semigroup of partial one-one transformations of $X_n$. A semigroup $S$ is called ample if it can be embedded in the symmetric inverse semigroup $I_n$ such that the image of $S$ is closed under the unary operation $\alpha \rightarrow \alpha \alpha^{-1}$ and $\alpha \rightarrow \alpha^{-1}\alpha $, where $\alpha^{-1}$ is the inverse of $\alpha$ in $I_n$. Recently, in 2023 Nasir and Umar showed that the dominion of ample subsemigroup $U$ of $I_n$ is the inverse subsemigroup of $I_n$ generated by $U$. [\cite{7}, Proposition 1].

A map $\alpha \in T_n$ is said to be order-decreasing if $(\forall x \in X_n) \  x\alpha \leq x$. The set
$$D_n=\{\alpha \in T_n: (\forall x\in X_n) \ x\alpha\leq x\},$$ 
is the subsemigroups of $T_n$ consisting of all order-decreasing maps of $X_n$. Umar in \cite{15} studied the structure of $D_n$ and characterized the Green's relations in it. Also, showed that $D_n$ is a non-regular abundant semigroup generated by its idempotents, and gave a formula to find the number of elements and idempotents in $D_n$.

A map $\alpha \in T_n$ is said to be  order-preserving if $(\forall \ x,y \in X_n)$ $x \leq y$ implies $x\alpha \leq y\alpha$. The set
$$O_n=\{\alpha \in T_n: (\forall \ x,y\in X_n) \ x \leq y \Rightarrow x\alpha\leq y\alpha\},$$  
is the subsemigroups of $T_n$ consisting of all  order-preserving maps of $X_n$. The semigroup $O_n$ was first studied by Aizenstat in 1962 \cite{AZ}, who gave a presentation for $O_n$. Howie in \cite{pod} studied the structure of $O_n$ and showed that $O_n$ is a regular idempotent-generated subsemigroup of $T_n$. Also, gave a formula to find the number of elements and idempotents in $O_n$. In \cite{GM}, Gomes and Howie established some properties of $O_n$.

The Catalan monoid $C_n$ is the subsemigroup of $T_n$ consisting of all maps that are both order-decreasing and order-preserving  $C_n=D_n \cap O_n$. Higgins in \cite{HIG} studied the structure of $C_n$ and gave a formula to find the number of elements and idempotents in $C_n$. In 1994 \cite{HI}, Higgins showed that $C_n$ is an idempotent-generated subsemigroup of $T_n$. Laradji and Umar investigated Further combinatorial properties for $C_n$ in \cite{8}. 

\section{Dominion of $D_n$ in $T_n$}
\vspace{0.3cm}
Let $E(D_n)$ be the set of idempotents in $D_n$. In order to study the dominion of $D_n$ in $T_n$ we need to characterize $\alpha^{\prime}$ in $T_n \setminus D_n$ so that $\alpha^{\prime}\alpha\alpha^{\prime}$ dominated by $D_n$.

\vspace{0.1cm}
\begin{proposition}\label{vx}
Let $\alpha$ $\in D_n$ and $\alpha^{\prime}$ $\in T_n \setminus D_n$ defined by
 \begin{equation*}
 \tag{1}
    x\alpha^{\prime}= \begin{cases}
    \text{min} (x\alpha^{-1}) & \text{if } x \in Im(\alpha)\\
    y \in[(z\alpha)\alpha^{-1}] \ \ \text{such that} \ \ z\alpha <x  & \text{if } x \notin Im(\alpha)
    \end{cases}
\end{equation*}\\
Then \\ 
1- $\alpha\alpha^{\prime}\alpha=\alpha$\\
2- $\alpha \alpha^{\prime} \in E(D_n)$\\
3- $\alpha^{\prime}\alpha \in E(D_n)$\\
4- $\alpha^{\prime} \alpha \alpha^{\prime} \in Dom_{T_{n}}(D_n)$

\end{proposition}  

\begin{proof}
Let $x \in X_n$. First, we will show that $x\alpha \alpha^{\prime}\alpha=x\alpha$. Since $x\alpha \in Im(\alpha)$ and $min((x\alpha)\alpha^{-1}) \in (x\alpha)\alpha^{-1}$, then 
$$x\alpha \alpha^{\prime}\alpha=(min((x\alpha)\alpha^{-1})\alpha=x\alpha,$$ so $\alpha \alpha^{\prime}\alpha=\alpha$.

For the second part, since $x \in (x\alpha)\alpha^{-1}$, then 
$$x\alpha \alpha^{\prime} =min((x\alpha)\alpha^{-1}) \leq x,$$
so $x\alpha \alpha^{\prime} \leq x$. Also
$(\alpha\alpha^{\prime})^2=\alpha\alpha^{\prime}$, thus $\alpha \alpha^{\prime} \in E(D_n)$.

Next, for the third part,
if $x \in Im(\alpha)$, since $min(x\alpha^{-1}) \in x\alpha^{-1}$, then $$x \alpha^{\prime}\alpha=(min(x\alpha^{-1}))\alpha=x.$$
If $x \notin Im(\alpha)$, then 
$$x \alpha^{\prime}\alpha=y\alpha=z\alpha<x,$$ so $x\alpha^{\prime}\alpha \leq x$. Also $(\alpha^{\prime}\alpha)^2=\alpha^{\prime}\alpha$, thus $\alpha^{\prime}\alpha \in E(D_n)$.

And now for the last part,let $f,g: T_n \rightarrow T$ be semigroup homomorphisms with $f|_{D_{n}}=g|_{D_{n}}$. \\ Then
\ \ \ \ \ 
$(\alpha^{\prime}\alpha\alpha^{\prime})f
\vspace{0.2cm}
= (\alpha^{\prime}\alpha)f(\alpha^{\prime})f
\\
\vspace{0.2cm}
\hspace{80pt}= (\alpha^{\prime}\alpha)g(\alpha^{\prime})f \hspace{42pt}[\alpha^{\prime}\alpha \in D_n]\\
\vspace{0.2cm}
\hspace{80pt}= (\alpha^{\prime})g(\alpha)g(\alpha^{\prime})f\\
\vspace{0.2cm}
\hspace{80pt}= (\alpha^{\prime})g(\alpha)f(\alpha^{\prime})f \hspace{30pt}[\alpha \in D_n]
\\
\vspace{0.2cm}
\hspace{80pt}= (\alpha^{\prime})g(\alpha\alpha^{\prime})f\\
\vspace{0.2cm}
\hspace{80pt}= (\alpha^{\prime})g(\alpha\alpha^{\prime})g \hspace{42pt}[\alpha \alpha^{\prime} \in D_n]
\\
\vspace{0.2cm}
\hspace{80pt}= (\alpha^{\prime}\alpha\alpha^{\prime})g$.\\[10pt]
Hence, $\alpha^{\prime} \alpha \alpha^{\prime} \in Dom_{T_n}(D_n)$.\\
\end{proof}

Next we characterize $\alpha \in D_n$ so that $\alpha^{\prime}\alpha\alpha^{\prime} \in Dom_{T_n}(D_n) \setminus D_n$.

\begin{lemma}
Let $\alpha \in D_n$ and $\alpha^{\prime} \in T_n \setminus D_n$ defined by (1).\\
1- If $\alpha \in E(D_n)$, then $\alpha^{\prime}\alpha\alpha^{\prime} \in D_n$.\\
2- If $\alpha \notin E(D_n)$, then $\alpha^{\prime}\alpha\alpha^{\prime} \in Dom_{T_n}(D_n) \setminus D_n$.
\end{lemma}

\begin{proof}
   1- Let $\alpha \in E(D_n)$ and $x \in X_n$. We will show that $x\alpha^{\prime}\alpha\alpha^{\prime} \leq x$. If $x \in Im(\alpha)$, then 
   $$x \alpha^{\prime}\alpha\alpha^{\prime} =(min(x\alpha^{-1}))\alpha\alpha^{\prime}=x\alpha^{\prime}=min(x\alpha^{-1})=x.$$
   If $x \notin Im(\alpha)$, then 
   $$x \alpha^{\prime}\alpha\alpha^{\prime} =y\alpha\alpha^{\prime}=z\alpha\alpha^{\prime}=(min((z\alpha)\alpha^{-1}))=z\alpha < x,$$
   thus    $\alpha^{\prime}\alpha\alpha^{\prime}\in D_n$.

\vspace{0.5cm}
2- Let $\alpha \notin E(D_n)$. Then $\exists \ x \in Im(\alpha)$ such that $x \notin x\alpha^{-1}$,
   so
   $$x \alpha^{\prime}\alpha\alpha^{\prime}=(min(x\alpha^{-1}))\alpha\alpha^{\prime}=x\alpha^{\prime}=min (x\alpha^{-1}).$$ Since $\forall \ y \in x\alpha^{-1}, y >x$, then $x\alpha^{\prime}\alpha\alpha^{\prime} > x$,
  and so $\alpha^{\prime}\alpha\alpha^{\prime} \in Dom_{T_n}(D_n) \setminus D_n$.\\
  \end{proof}

  \vspace{0.2cm}
  \begin{remark}
It is not guaranteed that $\alpha^{\prime}\alpha\alpha^{\prime}=\alpha^{\prime}$. For example,

\vspace{0.7cm}
let \
$\alpha=
\begin{pmatrix}  1&2&3&4&5\\1&1&3&2&2
\end{pmatrix}
$ $\in D_5$ and
$\alpha^{\prime}=
\begin{pmatrix}  1&2&3&4&5\\1&4&3&4&5
\end{pmatrix}$ $\in T_5$,
then  
$\alpha^{\prime}\alpha\alpha^{\prime}=
\begin{pmatrix}  1&2&3&4&5\\1&4&3&4&4
\end{pmatrix}
$$\neq \alpha^{\prime}$.
\end{remark}

\vspace{0.8cm}
In the following result we will see that for which $\alpha^{\prime}$ this equality will be true.

\vspace{0.1cm}
\begin{lemma}
Let $\alpha$ $\in D_n$ and $\alpha^{\prime}$ $\in T_n \setminus D_n$ defined by
 \begin{equation}
 \tag{2}
    x\alpha^{\prime}=
    \begin{cases}
    \text{min} (x\alpha^{-1}) & \text{if } x \in Im(\alpha)\\
    min((z\alpha)\alpha^{-1}) \ \ \text{such that} \ \ z\alpha <x  & \text{if } x \notin Im(\alpha)
    \end{cases}
\end{equation}

Then $\alpha^{\prime}\alpha\alpha^{\prime}=\alpha^{\prime}$.
 \end{lemma}

\begin{proof}
    Let $x \in X_n$. We will show that $x\alpha^{\prime}\alpha\alpha^{\prime}=x\alpha^{\prime}$. If $x \in Im(\alpha)$, then $$x\alpha^{\prime}\alpha\alpha^{\prime} =(min(x\alpha^{-1}))\alpha\alpha^{\prime}=x\alpha^{\prime}.$$
    If $x \notin Im(\alpha)$, then   $$x\alpha^{\prime}\alpha\alpha^{\prime}=(min((z\alpha)\alpha^{-1}))\alpha\alpha^{\prime}= z\alpha\alpha^{\prime}=(min((z\alpha)\alpha^{-1}))=x\alpha^{\prime},$$
    thus $\alpha^{\prime}\alpha\alpha^{\prime}=\alpha^{\prime}$.\\
\end{proof}

In what follows we show that it is only one element $\alpha^{\prime\prime} \in T_n$ that we need to add to $D_n$ to generate all elements of $Dom_{T_{n}}(D_n)$. Moreover, it will be shown that $Dom_{T_{n}}(D_n)$ is the smallest regular semigroup containing $D_n$.

\vspace{0.1cm}
\begin{proposition} \label{ndn}
Let $\alpha \in D_n$ defined by \\ 
\vspace{0.1cm}
\hspace{6cm} $x\alpha=max(1,x-1), \ \ \forall x\in X_n.$\\
and $\alpha^{\prime\prime}$ is given by\\
\vspace{0.1cm}
\hspace{6cm} $x\alpha^{\prime\prime}=min(x+1,n), \ \ \forall x>1, \ 1\alpha^{\prime\prime}=1$.\\ 
Then $Dom_{T_n}(D_n)=<D_n \bigcup \{{\alpha^{\prime\prime}}\}>$.
\end{proposition}

We observe that for all $x \in X_n$
$$1\alpha=2\alpha=1, \ \ \ \ \text{and} \ \ \ \  x\alpha=x-1 \ \ (x > 2).$$
Also, 
$$1\alpha^{\prime\prime}=1, \ \ \ \ x \alpha^{\prime\prime}=x+1 \ \ (1 < x < n), \ \ \ \  \text{and} \ \ \ \ n\alpha^{\prime\prime}=n.$$

\vspace{0.2cm}
Before proving this proposition, we need to prove several lemmas as follows.

\vspace{0.5cm}
Consider the semigroup \ $T_n^{(1)}=\{\alpha \in T_n : 1\alpha=1\}$, let \ $J^*_{r}=\{\alpha \in T_n^{(1)} : |Im(\alpha)|=r\}$, where $1 \leq r \leq n-1$, and let $E_{n-1}$ the set of idempotents in $J^*_{n-1}$.

\vspace{0.5cm}
We will adopt the proof of Lemma 6.3.2 in \cite{6} to prove the following lemma.

\begin{lemma}
 $J^*_{r}\subseteq \ <J^*_{n-1}>=T_n^{(1)}$. 
\end{lemma}

\begin{proof}
Let 

$$\alpha=
\begin{pmatrix}  A_1&A_2&A_3&...&A_r\\1&b_2&b_3&...&b_r
\end{pmatrix}
\in J^*_r,$$\\[5pt]
and let $b_i\alpha^{-1}=A_i \,\,\, (i=2,3,...,r)$, $ 1\alpha^{-1}=A_1$ and $1\in A_1$. Since not all of the sets $A_i$ are singletons, we may assume without loss of generality that $A_1=\{1,a_1,a_1^\prime,...\}$ has at least two elements. For all $x \in X_n$, we define $\epsilon: X_n \rightarrow X_n$ by
$$a_1\epsilon=a_1^\prime, \ \ \ \  \ \  x\epsilon=x \ \ (x \neq a_1),$$
and define $\beta: X_n \rightarrow X_n$ by
$$a_1\beta=b_{r+1} \ \ (b_{r+1} \notin Im(\alpha)), \ \ \ \ \ x\beta=1 \ \ (x \in A_1\setminus \{a_1\}), \ \ \ \ \ x\beta=b_i \ \ (x \in A_i, \ i\geq 2).$$
Then $\epsilon \in E_{n-1}$ and $|Im(\beta)|=|Im(\alpha)|+1$, so $\beta \in J^*_{r+1}$. Now we have that  
 $$a_1\epsilon\beta=1, \ \ \ \ \ x\epsilon\beta=1 \ \ (x \in A_1\setminus \{a_1\}), \ \ \ \ \ x\epsilon\beta=b_i \ \ (x \in A_i, \ i\geq 2).$$
Thus for all $x\in X_n$, $x\epsilon\beta=x\alpha$, and so $\alpha=\epsilon\beta$.\\
\end{proof}
As a consequence of this lemma, we deduce that $T_n^{(1)}=<J^*_{n-1}>$.

\newpage
For any $\alpha$ in $T_n$, a cycle for $\alpha$ is a set of elements $\{x_1,x_2,x_3...,x_k\} \subseteq X_n$ such that

\begin{center}
    $x_1 \rightarrow x_2,\ x_2 \rightarrow x_3,..., x_k \rightarrow x_1$.
\end{center}
 The length of a cycle is the number of elements in the cycle, for the cycle $\{x_1,x_2,x_3...,x_k\}$, the length is $k$. The least common multiple $(lcm)$ of the lengths of the cycles is the smallest number $t$ such that $\alpha^{t}(x)=x$ for all $x$ that are part of cycles \cite{50}.

\vspace{0.5cm}
\begin{remark}
If $\alpha$ is an idempotent in $J^*_{n-1}$, then there is only one $x$ in $X_n$ such that $x\alpha <x$ \ or \  $x\alpha>x$, so every idempotent in $J^*_{n-1}$ either increasing or decreasing.
\end{remark}

\vspace{0.5cm}
\begin{lemma}
  $J^*_{n-1}\subseteq \ <D_n \bigcup \{\alpha^{\prime\prime}\}>$.
\end{lemma}

\begin{proof}
  We need to prove that $J^*_{n-1}
  \subseteq \ <D_n \bigcup \{\alpha^\prime\}>$ $\subseteq \ <D_n \bigcup \{\alpha^{\prime\prime}\}>$. 
Let $\alpha$ and $\alpha^{\prime\prime}$ defined by Proposition \ref{ndn},

\vspace{0.2cm}
\begin{center}
$\alpha=
\begin{pmatrix}  
1&2&3&4&...&n-1&n\\1&1&2&3&...&n-2&n-1
\end{pmatrix}$ and
$\alpha^{\prime\prime}=
\begin{pmatrix}  
1&2&3&4&...&n-1&n\\1&3&4&5&...&n&n
\end{pmatrix}$\\[10pt]
\end{center}
and let $\alpha^{\prime}_y$ be a bijection from $\{2,3,...,n-1\}$ to $\{3,4,...,n\}$, also $min(1\alpha^{-1})=1$, $min(x\alpha^{-1})=x+1 \ (2\leq x <n$), and \ $n\alpha^{\prime}_y=y\in \{1,3,4,...,n\}$, so

\vspace{0.3cm}
$$\alpha^{\prime}_y=
\begin{pmatrix}  
1&2&3&4&...&n-1&n\\1&3&4&5&...&n&y
\end{pmatrix}.$$

\vspace{0.5cm}
Let $f_{u,v}$ be the decreasing idempotent maps in $J^*_{n-1}$, then

\vspace{0.2cm}
$$f_{u,v}=
\begin{pmatrix}  
1&2&...&u-1&u&u+1&...&n\\1&2&...&u-1&v&u+1&...&n
\end{pmatrix}, \ \ \ (1\leq v <u \leq n).$$\\[10pt]
We see that $f_{u,v} \in \ <D_n \bigcup \{\alpha^\prime\}>$. Let $g_{r,s}$ be the increasing idempotent maps in $J^*_{n-1}$, then

\vspace{0.2cm}
$$g_{r,s}=
\begin{pmatrix}  
1&2&...&r&r+1&r+2&...&n\\1&2&...&s&r+1&r+2&...&n
\end{pmatrix}, \ \ \ (1<r<s \leq n).$$\\[10pt]
We want to generate $g_{r,s}$ by some $f_{u,v}$ and $\alpha^{\prime}$,
and we will do that by considering two cases.\\[10pt]
Case (1) If $r=2$, then $g_{2,s}$ has four cases depending on $s$. Now we will study them one by one.

\vspace{0.3cm}
For $s=3$, then

$$f_{3,2} \alpha^{\prime}_{4}=
\begin{pmatrix}  
1&2&3&4&...&n-1&n\\1&3&3&5&...&n&4
\end{pmatrix}$$

have cycles of lengths $1$, $1$ and $n-3$, $( \{1,1\}, \{3,3\},$ and $\{4,5,...,n-1,n\},$ resp.). Since $2$ linked to the cycle $\{3,3\}$ of length $1$, then 
$f_{3,2} \alpha^{\prime}_{4}(2)=$$(f_{3,2} \alpha^{\prime}_{4})^{lcm(1,n-3)}(2)=3$, and so $(f_{3,2} \alpha^{\prime}_{4})^{(n-3)}=g_{2,3}$.

\vspace{0.5cm}
For $3<s<n-1$,
then $$f_{s,2} \alpha^{\prime}_{s+1}=
\begin{pmatrix}  
1&2&3&...&s-1&s&s+1&...&n-1&n\\1&3&4&...&s&3&s+2&...&n&s+1
\end{pmatrix}$$
have cycles of lengths $1$, $s-2$ and $n-s$, $( \{1,1\}, \{3,4,...,s-1,s\},$ and $\{s+1,s+2,...,n-1,n\}$, resp.). Since $2$ linked to the cycle $\{3,4,...,s-1,s\}$ of length $s-2$ and $(f_{s,2} \alpha^{\prime}_{s+1})^{(s-2)}$ maps every element in this cycle to itself, then $(f_{s,2} \alpha^{\prime}_{s+1})^{(s-2)}(2)=$$(f_{s,2} \alpha^{\prime}_{s+1})^{lcm(1,s-2,n-s)}(2)=s$, and so $(f_{s,2} \alpha^{\prime}_{s+1})^{lcm(s-2,n-s)}=g_{2,s}$.

\vspace{0.5cm}
For $s=n-1$,
then $$f_{n-1,2} \alpha^{\prime}_{n}=
\begin{pmatrix}  
1&2&3&...&n-2&n-1&n\\1&3&4&...&n-1&3&n
\end{pmatrix}$$
have cycles of lengths $1$, $n-3$ and $1$, $( \{1,1\}, \{3,4,...,n-2,n-1\},$ and $\{n,n\}$, resp.). Since $2$ linked to the cycle $\{3,4,...,n-2,n-1\}$ of length $n-3$ and $(f_{n-1,2} \alpha^{\prime}_{n})^{(n-3)}$ maps every element in this cycle to itself, then $(f_{n-1,2} \alpha^{\prime}_{n})^{(n-3)}(2)=(f_{n-1,2} \alpha^{\prime}_{n})^{lcm(1,n-3)}(2)=n-1$, and so $(f_{n-1,2} \alpha^{\prime}_{n})^{(n-3)}=g_{2,n-1}$.

\vspace{0.5cm}
For $s=n$,
then $$f_{n,2} \alpha^{\prime}_{1}=
\begin{pmatrix}  
1&2&3&...&n-1&n\\1&3&4&...&n&3
\end{pmatrix}$$ 
have cycles of lengths $1$ and $n-2$, $( \{1,1\},$ and $\{3,4,...,n-1,n\}$, resp.). Since $2$ linked to the cycle $\{3,4,...,n-1,n\}$ of length $n-2$ and $(f_{n,2} \alpha^{\prime}_{1})^{(n-2)}$ maps every element in this cycle to itself, then $(f_{n,2} \alpha^{\prime}_{1})^{(n-2)}(2)=(f_{n,2} \alpha^{\prime}_{1})^{lcm(1,n-2)}(2)=n$, and so $(f_{n,2} \alpha^{\prime}_{1})^{(n-2)}=g_{2,n}$.

\vspace{1cm}
Case (2) If $r>2$, then $g_{r,s}$ has three cases depending on $s$.

\vspace{0.5cm}
For $3<s<n-1$, then 
$$f_{s,r} \alpha^{\prime}_{s+1}f_{r,2}=
\setcounter{MaxMatrixCols}{14}
\begin{pmatrix}
1&2&3&...&r-1&r&r+1&...&s-1&s&s+1&...&n-1&n\\1&3&4&...&2&r+1&r+2&...&s&r+1&s+2&...&n&s+1
\end{pmatrix}$$
have cycles of lengths $1$, $r-2$, $s-r$ and $n-s$, $( \{1,1\}, \{2,3,4,...,r-1\}, \{r+1,r+2,...,s-1,s\}$ and $\{s+1,s+2,...,n-1,n\}$, resp.). Since $r$ linked to the cycle $\{r+1,r+2,...,s-1,s\}$ of length $s-r$ and $(f_{s,r} \alpha^{\prime}_{s+1}f_{r,2})^{(s-r)}$ maps every element in this cycle to itself, then $(f_{s,r} \alpha^{\prime}_{s+1}f_{r,2})^{(s-r)}(r)=(f_{s,r} \alpha^{\prime}_{s+1}f_{r,2})^{lcm(r-2,s-r,n-s)}(r)=s$, and so $(f_{s,r} \alpha^{\prime}_{s+1}f_{r,2})^{lcm(1,r-2,s-r,n-s)}=g_{r,s}$.

\vspace{0.5cm}
For $s=n-1$, then

$$f_{n-1,r} \alpha^{\prime}_{n}f_{r,2}=
\begin{pmatrix}  
1&2&3&...&r-1&r&r+1&...&n-2&n-1&n\\1&3&4&...&2&r+1&r+2&...&n-1&r+1&n
\end{pmatrix}$$

have cycles of lengths $1$, $r-2$, $n-1-r$ and $1$,  $( \{1,1\}, \{2,3,4,...,r-1\}, \{r+1,r+2,...,n-2,n-1\}$ and $\{n,n\}$, resp.). Since $r$ linked to the cycle $\{r+1,r+2,...,n-2,n-1\}$ of length $n-1-r$ and $(f_{n-1,r} \alpha^{\prime}_{n}f_{r,2})^{(n-1-r)}$ maps every element in this cycle to itself, then $(f_{n-1,r} \alpha^{\prime}_{n}f_{r,2})^{(n-1-r)}(r)=(f_{n-1,r} \alpha^{\prime}_{n}f_{r,2})^{lcm(1,r-2,n-1-r)}(r)=n-1$, and so $(f_{n-1,r} \alpha^{\prime}_{n}f_{r,2})^{lcm(1,r-2,n-1-r)}=g_{r,n-1}$.

\vspace{0.5cm}
For $s=n$, then 
$$f_{n,r} \alpha^{\prime}_{1}f_{r,2}=
\begin{pmatrix}  
1&2&3&...&r-1&r&r+1&...&n-1&n\\1&3&4&...&2&r+1&r+2&...&n&r+1
\end{pmatrix}$$
have cycles of lengths $1$, $r-2$ and $n-r$, $( \{1,1\}, \{2,3,4,...,r-1\},$ and $\{r+1,r+2,...,n-1,n\}$, resp.). Since $r$ linked to the cycle $\{r+1,r+2,...,n-1,n\}$ of length $n-r$ and $(f_{n,r} \alpha^{\prime}_{1}f_{r,2})^{(n-r)}$ maps every element in this cycle to itself, then $(f_{n,r} \alpha^{\prime}_{1}f_{r,2})^{(n-r)}(r)=(f_{n,r} \alpha^{\prime}_{1}f_{r,2})^{lcm(1,r-2,n-r)}(r)=n$,
and so $(f_{n,r} \alpha^{\prime}_{1}f_{r,2})^{lcm(1,r-2,n-r)}=g_{r,n}$.

\vspace{0.5cm}
From case (1) and case (2), we conclude that 

\begin{equation*}
    g_{r,s}= \begin{cases}
     (f_{s,2}\alpha^{\prime}_{s+1})^{t} & \text{if } r=2\\
    (f_{s,r}\alpha^{\prime}_{s+1}f_{r,2})^{t} & \text{if } r>2
    \end{cases}
\end{equation*}\\[10pt]
where $t=lcm(s-r,n-s,r-2)$ and $s \neq r \neq n$. Therefore, $J^*_{n-1}\subseteq \ <D_n \bigcup \{\alpha^\prime\}>$.

\vspace{0.9cm}
Now we want to generate $\alpha^\prime$ by some $f_{u,v}$ and $\alpha^{\prime\prime}$. Let

\vspace{0.5cm}
$$f_{n,z}=
\begin{pmatrix}  
1&2&3&4&...&n-1&n\\1&2&3&4&...&n-1&z
\end{pmatrix} \ \ \ (1\leq z <n).$$

Then

\vspace{0.5cm}
\begin{equation*}
f_{n,z}\alpha^{\prime\prime}=
    \begin{cases}
    \begin{pmatrix}  
1&2&3&4&...&n-1&n\\1&3&4&5&...&n&1
\end{pmatrix}=\alpha^{\prime}_1 & \text{if } z=1\\
\\
    \begin{pmatrix}  
1&2&3&4&...&n-1&n\\1&3&4&5&...&n&z+1
\end{pmatrix}=\alpha^{\prime}_{z+1}  & \text{if } 2\leq z <n-1\\
\\
    \begin{pmatrix}  
1&2&3&4&...&n-1&n\\1&3&4&5&...&n&n
\end{pmatrix}=\alpha^{\prime}_n=\alpha^{\prime\prime}  & \text{if } z=n-1
    \end{cases}
\end{equation*}\\[10pt]
Thus $\alpha^{\prime}_y=f_{n,z}\alpha^{\prime\prime}$ and so $<D_n \bigcup \{\alpha^{\prime}\}>$ $\subseteq \ <D_n \bigcup \{\alpha^{\prime\prime}\}>$.\\
\end{proof}

\begin{lemma}
  Let $\beta \in T_n$ be such that $1\beta>1$. Then $\beta \notin <D_n \bigcup \{\alpha^{\prime\prime}\}>$.
\end{lemma}

\begin{proof}
   Let $\beta \in T_n$. Since $\forall \ \alpha\in D_n$, $1\alpha=1$ and $1\alpha^{\prime\prime}=1$, then $\beta \notin <D_n \bigcup \{{\alpha^{\prime\prime}}\}>$.
\end{proof}

\vspace{0.5cm}
\begin{corollary}\label{nmn}
  $<D_n \bigcup \{\alpha^{\prime\prime}\}>=T_n^{(1)}.$  
\end{corollary}

\vspace{0.5cm}
Now to prove Proposition \ref{ndn} , we observe that every elements of $Dom_{T_n}(D_n)$ belongs to $T_n^{(1)}$, then by Corollary \ref{nmn} we have that
$$Dom_{T_n}(D_n) \subseteq \ <D_n \bigcup \{{\alpha^{\prime\prime}}\}>$$

\vspace{0.5cm}
For the second inclusion, we want to prove that $\alpha\alpha^{\prime\prime}\alpha, \ \alpha\alpha^{\prime\prime},\ \alpha^{\prime\prime}\alpha$ and $\alpha^{\prime\prime}$ in $Dom_{T_n}(D_n)$.

\vspace{0.5cm}
For $\alpha\alpha^{\prime\prime}$ \ and \ $\alpha\alpha^{\prime\prime}\alpha$.\\ 
\vspace{0.1cm}
 If $x=1$, then $1\alpha\alpha^{\prime\prime}=1\alpha^{\prime\prime}=1$, \ and so \ $1\alpha\alpha^{\prime\prime}\alpha=1\alpha$.\\ 
\vspace{0.1cm}
If $x=2$, then $2\alpha\alpha^{\prime\prime}=1\alpha^{\prime\prime}=1$, \ and so \  $2\alpha\alpha^{\prime\prime}\alpha=1\alpha=2\alpha$.\\
\vspace{0.1cm}
If $x>2$, then $x\alpha\alpha^{\prime\prime}=(x-1)\alpha^{\prime\prime}=x$, \ and so \ $x\alpha\alpha^{\prime\prime}\alpha=x\alpha$.\\ Thus \
$x\alpha\alpha^{\prime\prime} \leq x$, so
$\alpha\alpha^{\prime\prime} \in D_n$ and $\alpha\alpha^{\prime\prime}\alpha=\alpha \in D_n$.

\vspace{0.5cm}
For $\alpha^{\prime\prime}\alpha$ \ and \ $\alpha^{\prime\prime}$.\\
\vspace{0.1cm}
\ If $x=1$, then $1\alpha^{\prime\prime}\alpha=1\alpha=1$, \ and so \ $1\alpha^{\prime\prime}\alpha\alpha^{\prime\prime}=1\alpha^{\prime\prime}$.\\
    \vspace{0.1cm}
\ If $x=n$, then $n\alpha^{\prime\prime}\alpha=n\alpha=n-1$, \ and so \ $n\alpha^{\prime\prime}\alpha\alpha^{\prime\prime}=(n-1)\alpha^{\prime\prime}=n=n\alpha^{\prime\prime}$.\\
\vspace{0.1cm}
\ If $1<x<n$, then $x\alpha^{\prime\prime}\alpha=(x+1)\alpha=x$, \ and so \ $x\alpha^{\prime\prime}\alpha\alpha^{\prime\prime}=x\alpha^{\prime\prime}$.\\ Thus \
$x\alpha^{\prime\prime}\alpha\ \leq x$, so
$\alpha^{\prime\prime}\alpha\in D_n$, and by Proposition \ref{vx}, $\alpha^{\prime\prime}=\alpha^{\prime\prime}\alpha\alpha^{\prime\prime} \in Dom_{T_n}(D_n)$. Therefore, 
$$<D_n \bigcup \{{\alpha^{\prime\prime}}\}> \ \subseteq Dom_{T_n}(D_n).$$

\vspace{0.5cm}
From the above, we deduce that $Dom_{T_n}(D_n) = \ <D_n \bigcup \{{\alpha^{\prime\prime}}\}> \ =T_n^{(1)}.$

\vspace{1cm}
Now we want to prove that $Dom_{T_{n}}(D_n)$ is regular. Let $\alpha \in T^{(1)}_n$ and define a mapping $\beta:X_n \rightarrow X_n$ as follows:
    
    \begin{equation*}
    x\beta= \begin{cases}
     min(x\alpha^{-1}) & \text{if} \ \ x\in Im(\alpha) \\
    z\in X_n  & \text{if} \ \ x\notin Im(\alpha)
    \end{cases}
\end{equation*}

\vspace{0.8cm}
 Then for all $x \in X_n,$ 
 $$x\alpha\beta\alpha=min((x\alpha)\alpha^{-1})\alpha=x\alpha,$$
 
 it follows that $\alpha\beta\alpha=\alpha$.
 Since $1 \in Im(\alpha)$ , then $1\beta=min(1\alpha^{-1})=1$, so $\beta \in T^{(1)}_n.$ Hence $T^{(1)}_n$ is regular.

\vspace{0.5cm}
In the next example, we will apply Proposition \ref{ndn} to find the dominion of $D_3$ in $T_3$.

\vspace{0.3cm}
\begin{example}
Consider the chain $X_3=\{1,2,3\}$. Let
$$\alpha=\begin{pmatrix}  1&2&3\\1&1&2
\end{pmatrix}\in D_3 \ \ \ \text{and} \ \ \ \alpha^{\prime\prime}=\begin{pmatrix}  1&2&3\\1&3&3
\end{pmatrix}\in T_3.$$\\[10pt]
Then $Dom_{T_3}(D_3)=<D_3 \bigcup \{{\alpha^{\prime\prime}}\}>=$  

\vspace{0.5cm}
$\Bigg\{
   \begin{pmatrix}  1&2&3\\1&2&3
\end{pmatrix}, 
\begin{pmatrix}  1&2&3\\1&2&2
\end{pmatrix},
\begin{pmatrix}  1&2&3\\1&2&1
\end{pmatrix},
\begin{pmatrix}  1&2&3\\1&1&2
\end{pmatrix},
\begin{pmatrix}  1&2&3\\1&1&3
\end{pmatrix},
\begin{pmatrix}  1&2&3\\1&1&1
\end{pmatrix},
\begin{pmatrix}  1&2&3\\1&3&3
\end{pmatrix},
\begin{pmatrix}  1&2&3\\1&3&1
\end{pmatrix}
\Bigg\}$.
\end{example}

\vspace{0.5cm}
\begin{lemma}  
$Dom_{T_n}(D_n)$ is idempotent-generated.
\end{lemma}

\begin{proof}
  As we known that $D_n$ is idempotent-generated \cite{15},
and 
\vspace{0.5cm}
$$\alpha^{\prime\prime}=
\begin{pmatrix}  
1&2&3&4&...&n-1&n\\1&3&4&5&...&n&n
\end{pmatrix} \notin D_n.$$

 We need to show that $\alpha^{\prime\prime}$ is a product of idempotents in $T_n^{(1)}.$ For all $x \in X_n$ we define $\epsilon_{i}: X_n \rightarrow X_n$, where $i=2,3,4,...,n-1$, by
 
$$ i\epsilon_{i}=i+1, \ \ \ \ \ \ \ \ \ x\epsilon_{i}=x \ \ (x \neq i).$$
Then $\epsilon_{i}$ are idempotents in $T_n^{(1)}$. Now let $\beta=\epsilon_{n-1}\epsilon_{n-2}...\epsilon_{3}\epsilon_{2}$, then

$$1\beta=1, \ \ \ \ \ \ \ \ \  \ \ \ x\beta=x+1 \ \ (2 \leq x \leq n-1), \ \ \ \ \ \ \ \ \ \ \ \ n\beta=n,$$
it is not difficult to see that

$$\alpha^{\prime\prime}=\epsilon_{n-1}\epsilon_{n-2}...\epsilon_{3}\epsilon_{2},$$
and so $Dom_{T_n}(D_n)=<D_n \bigcup \{{\alpha^{\prime\prime}}\}>$ is idempotent-generated.
\end{proof}

\vspace{0.5cm}
The following lemma gives a formula to find the number of elements of dominion $D_n$ in $T_n$.

\vspace{0.2cm}
\begin{lemma}  
$|Dom_{T_n}(D_n)|= n^{n-1}-(n-1)!+1$.
\end{lemma}

\begin{proof}
    Let $\alpha \in T_n^{(1)}$ and defined $\bar{\alpha}$ by 
    $$x\bar{\alpha}=(x+1)\alpha \ \ \ (\forall x \in X_{n-1}, \ \bar{\alpha} \in T_{n-1} \setminus S_{n-1}).$$
    Since $1\alpha=1$ and every $x$ has $n$ degrees of freedom, we have $n^{n-1}$ possible maps, however we need to remove $(n-1)!$ permutations where $1\alpha=1$ but not the identity map.
\end{proof}

Here are the cardinalities of $Dom_{T_n}(D_n)$ for $n=1,2,...,8$

\vspace{0.5cm}
\begin{center}
    \begin{tabular}{|c|c|}
    \hline
    $n$  & $|Dom_{T_n}(D_n)|$ \\
    \hline
    1  & 1 \\
        \hline
       2  & 2 \\
       \hline
       3  & 8 \\
        \hline
       4  & 59 \\
        \hline
       5  & 602  \\
        \hline
       6  &  7657\\
        \hline
       7  &  116930\\
        \hline
        8 & 2092113\\
        \hline
    \end{tabular}
\end{center}
\vspace{0.5cm}
The sequence: $1,2,8,59,602,7657,...$ is not yet recorded in \cite{seq}.

\vspace{0.9cm}
Although the following formula is known and can be found in \cite{seq}, we will present it in a separate lemma and prove it.
\begin{lemma}  
$|E(Dom_{T_n}(D_n))|=\sum_{k=1}^{n}k^{(n-k)}$
$\begin{pmatrix} n-1 \\ k-1 \end{pmatrix}.$
\end{lemma}

\begin{proof}
  Since $1\alpha=1$ is a fixed point, we choose the remaining $k-1$ fixed points from $X_{n-1}$ in $\begin{pmatrix} n-1 \\ k-1 \end{pmatrix}$ ways.\\
  Next, we see that the remaining $n-k$ elements can be mapped to any of the $k$ chosen fixed points in $k^{(n-k)}$ ways.\\
  Multiplying these numbers we get $k^{(n-k)} \begin{pmatrix} n-1 \\ k-1 \end{pmatrix}$, the number of idempotents of rank $k$. Finally taking the sum over $k$ from $1$ to $n$ yields the required result.\\
\end{proof}

Here are the cardinalities of $E(Dom_{T_n}(D_n))$ for $n=1,2,...,8$

\vspace{0.5cm}
\begin{center}
    \begin{tabular}{|c|c|}
    \hline
    $n$  & $|E(Dom_{T_n}(D_n))|$ \\
    \hline
    1  & 1 \\
        \hline
       2  & 2 \\
       \hline
       3  & 6 \\
        \hline
       4  & 23 \\
        \hline
       5  & 104  \\
        \hline
       6  &  537\\
        \hline
       7  &  3100\\
        \hline
        8 & 19693\\
        \hline
    \end{tabular}
\end{center}
\vspace{0.5cm}
The sequence: $1,2,6,23,104,537,...$ is in \cite{seq}.

\section{Dominion of $O_n$ in $T_n$}
\vspace{0.6cm}
In this section, we characterize the smallest regular semigroup containing $O_n$ to conclude that $O_n$ is closed in $T_n$. Before that, we need to prove some results.

\vspace{0.5cm}
\begin{lemma}
     Let $\alpha \in O_n$ and $\beta \in T_n$ such that $\alpha \beta \alpha=\alpha$. Then $\alpha \beta \in O_n$. 
    \end{lemma}
    
     \begin{proof}
Let
$$\alpha=
\begin{pmatrix} A_1&A_2&...&A_r\\a_1&a_2&...&a_r
\end{pmatrix} \in O_n ,$$
where $A_i=a_i\alpha^{-1}$ $(1\le i \le r)$. Let $a_i \le a_j$. Since $\alpha \beta \alpha=\alpha$, then 
$$a_i\beta \in a_i \alpha^{-1}=A_i \le A_j=a_j\alpha^{-1} \ \ \ \text{and} \ \ \  a_j\beta \in a_j\alpha^{-1}.$$
So $a_i\beta \le a_j\beta$.
Thus $\beta|_{Im(\alpha)}$ is order preserving and so $\alpha\beta \in O_n.$ 
\end{proof}

\vspace{0.5cm}
\begin{proposition}
Let $\alpha \in O_n$ and $\beta \in T_n$ such that $\alpha\beta, \ \beta\alpha \in O_n$, then $\beta\alpha\beta$ in $O_n$. 
\end{proposition}  

\begin{proof}
Let $x,y \in X_n$ and $x\leq y$. Then we want to show that $x\beta\alpha\beta \leq y\beta\alpha\beta$. If $x\beta \leq y\beta$, since $\alpha\beta \in O_n$ then
$$(x\beta)\alpha\beta \leq (y\beta)\alpha\beta.$$
If $x\beta \geq y\beta$, since $\alpha \in O_n$, we see that $x\beta\alpha \geq y\beta\alpha.$
However, since $\beta\alpha \in O_n$, then $x\beta\alpha \leq y\beta\alpha.$
thus $x\beta\alpha=y\beta\alpha.$
Hence $$x\beta\alpha\beta=y\beta\alpha\beta.$$
\end{proof}

\vspace{0.2cm}
The following result describes the dominion of a finite regular semigroup.
\begin{theorem}(\cite{196}, Theorem 5)
 Let $U$ be any finite regular semigroup and let $S$ be any finite semigroup containing $U$ as a proper subsemigroup. Then the dominion of $U$ in $S$ is strictly contained in $S$.
\end{theorem}

\vspace{0.5cm}
Since $O_n$ is a finite regular subsemigroup of $T_n$, then  
the dominion of $O_n$ in $T_n$ is not all of $T_n$.
\begin{lemma}
    The semigroup $O_n$ is saturated, $Dom_{T_n}(O_n) \neq T_n$.
\end{lemma}

\vspace{0.4cm}
A map $\alpha \in T_n$ is said to be  order-reversing if $(\forall \ x,y \in X_n)$ $x \leq y$ implies $x\alpha \geq y\alpha$. The set
$$OR_n=\{\alpha \in T_n: (\forall \ x,y\in X_n) \ x \leq y \Rightarrow x\alpha\geq y\alpha\},$$  
is the subsemigroups of $T_n$ consisting of all  order-reversing maps of $X_n$ \cite{cla}. 

\vspace{0.2cm}
Let
$$\beta_{n}=\begin{pmatrix}  1&2&3&...&n-1&n\\2&1&1&...&1&1
\end{pmatrix} \in OR_n,$$

\newpage
in the next result we characterize the subsemigroup $<O_n \bigcup \{\beta_n\}>$ of $T_n$.
\begin{lemma}
    $<O_n \bigcup \{\beta_n\}>=O_n \cup \{\alpha \in OR_n: |Im(\alpha)|=2\}$.
\end{lemma}
\begin{proof}
  We want to prove that every element of $OR_n$ with image size $2$ can be generated by elements of $O_n$ and $\beta_n$.\\[10pt]
  Let 
  $$\beta_{n,i}=\begin{pmatrix}  \{1,2,...i\}&...&\{i+1,...,n\}\\2&...&1
  \end{pmatrix} \in OR_n \ \ \text{and} \ \ |Im(\beta_{n,i})|=2.$$\\[10pt]
  Then 
  $$\begin{pmatrix}  \{1,2,...i\}&...&\{i+1,...,n\}\\1&...&2
  \end{pmatrix}\beta_n=\beta_{n,i}  \ \ \ \text{for all} \ \ i,$$\\[10pt]
  and  
  $$\beta_{n,i}\begin{pmatrix}  1&2&3&...&n\\x&y&y&...&y
  \end{pmatrix}=\begin{pmatrix}  \{1,2,...i\}&...&\{i+1,...,n\}\\y&...&x
  \end{pmatrix} \in OR_n \ \ \  \text{where} \ \ \  1\leq x<y\leq n.$$\\
\end{proof}

Now we want to show that any subsemigroup of $T_n$ that contains $O_n$ must contain $<O_n \bigcup \{\beta_n\}>$.

\vspace{0.5cm}
\begin{lemma}
$<O_n \bigcup \{\beta_n\}>$ is the smallest regular subsemigroup of $T_n$ containing $O_n$.
\end{lemma}

\begin{proof}
   Let $\alpha \notin <O_n \bigcup \{\beta_n\}>$, we want to prove that $<O_n \bigcup \{\beta_n\}>$ contained in $<O_n \bigcup \{\alpha\}>$.\\[10pt] 
   Let  
   
   $$\alpha=\begin{pmatrix}  1&2&3&4&...&n\\2&1&2&x&...&x
  \end{pmatrix} \ \ \ \text{where} \ \ \ 1<x<n.$$\\[10pt]
  Then
  $$\beta_n=
  \begin{pmatrix}  1&2&3&4&...&n\\1&2&2&2&...&2
  \end{pmatrix}\begin{pmatrix}  1&2&3&4&...&n\\2&1&2&x&...&x
  \end{pmatrix}.$$\\[10pt]
   Now let 
  $$\alpha=\begin{pmatrix}  1&2&3&4&...&n\\x&y&x&z&...&z
  \end{pmatrix} \ \ \ \text{where} \ \ 1<y\leq x-1<x.$$ 
  Then

  \vspace{0.7cm}
  $\beta_n=
  \begin{pmatrix}  1&2&3&4&...&n\\1&2&2&2&...&2
  \end{pmatrix}\begin{pmatrix}  1&2&3&4&...&n\\x&y&x&z&...&z
  \end{pmatrix}\begin{pmatrix}  1&2&3&...&x-1&x&x+1&...&n\\1&1&1&...&1&2&2&...&2
  \end{pmatrix}.$\\[10pt]
Thus $\beta_n \in <O_n \bigcup \{\alpha\}>$, and so $<O_n \bigcup \{\beta_n\}>$ is the intersection of all subsemigroups of $T_n$ which contain $O_n$.

  \vspace{0.5cm}
 Now we know that the semigroup $O_n$ is regular. Let 
 
 $$\alpha=\begin{pmatrix}  \{1,2,...i\}&...&\{i+1,...n\}\\y&...&x
  \end{pmatrix} \in OR_n \ \ \ \text{and} \ \ \ |Im(\alpha)|=2, \ \ \ \text{where} \ \ 1\leq x<y \leq n.$$ 
  
 Define 
  $$\beta=\begin{pmatrix}  \{1,2,...x\}&...&\{x+1,...n\}\\i+1&...&i
  \end{pmatrix}.$$\\[6pt] 
  Then $\beta \in OR_n$,  $|Im(\beta)|=2$, and $\alpha\beta\alpha=\alpha$. Hence  $<O_n \bigcup \{\beta_n\}>$ is regular.\\
\end{proof}

The following result will be used in the proof of the next proposition
\begin{theorem}(Isbell's Zigzag Theorem [\cite{RI}, Theorem 2.3]) \label{ZT}
    Let $U$ be a subsemigroup of $S$ and let $d \in S$. Then $d \in Dom_S(U)$ if and only if $d \in U$ or there exists a series of factorizations of $d$ as follows:
    \begin{center}
$d=u_0y_1=x_1u_1y_1=x_1u_2y_2=x_2u_3y_2=...=x_mu_{2m-1}y_m=x_mu_{2m}$,\\
    \end{center}
    where $u_i\in U, \ \ x_i,y_i\in S, \ \ u_0=x_1u_1, \ \ u_{2i-1}yi=u_{2i}y_{i+1}, \ \ x_iu_{2i}=x_{i+1}u_{2i+1} \ (1\leq i \leq m-1)$ and \ $u_{2m-1}y_m=u_{2m}$.
\end{theorem}
Such equations are known as a zigzag in $S$ over $U$ with value $d$, length $m$, and spine $u_0,u_1,...,
u_{2m}$ (in that order).\\

\vspace{0.5cm}
Now we characterize the dominion of $O_n$ in $T_n$.
\begin{proposition}
  The semigroup $O_n$ is closed in $T_n$, $Dom_{T_n}(O_n)=O_n$. 
\end{proposition}

\begin{proof}
 Suppose that 
 $$\beta_{n}=\begin{pmatrix}  1&2&3&...&n-1&n\\2&1&1&...&1&1
\end{pmatrix}\in Dom_{T_n}(O_n).$$\\[3pt]
So by Theorem \ref{ZT}, there exists a series of factorizations of $\beta_n$ as follows:
$$\beta_n=u_0y_1=x_1u_1y_1=x_1u_2y_2=x_2u_3y_2=...=x_mu_{2m-1}y_m=x_mu_{2m},$$
where $u_i\in O_n, x_i,y_i\in T_n, \ \ u_0=x_1u_1, \ \ u_{2i-1}yi=u_{2i}y_{i+1}, \ \ x_iu_{2i}=x_{i+1}u_{2i+1} \ (1\leq i \leq m-1)$ and \ $u_{2m-1}y_m=u_{2m}$.
Let $x_1 \in T_n  \setminus O_n$, then $x_1$ must be an order-reversing map and $|Im(x_1)|=2$. So we need to define $u_1 \in O_n$ such that $u_0=x_1u_1$, where $u_0 \in O_n$. However, the product of an order-reversing map and an order-preserving map is order reversing, so $x_1u_1 \neq u_0$. Therefore, it is impossible to have a zigzag in $T_n$ over $O_n$ with $\beta_n$; this is a contradiction. Thus $\beta_n \notin Dom_{T_n}(O_n)$, and so $O_n$ is closed in $T_n$. 
\end{proof}

\vspace{0.5cm}
\begin{example}
Consider the chain $X_3=\{1,2,3\}$. The dominion of $O_3$ in $T_3$ is 

\vspace{0.5cm}
$Dom_{T_3}(O_3)= \Bigg\{
   \begin{pmatrix}  1&2&3\\1&2&3
\end{pmatrix}, 
\begin{pmatrix}  1&2&3\\1&2&2
\end{pmatrix},
\begin{pmatrix}  1&2&3\\1&3&3
\end{pmatrix},
\begin{pmatrix}  1&2&3\\2&3&3
\end{pmatrix},
\begin{pmatrix}  1&2&3\\1&1&2
\end{pmatrix},
\begin{pmatrix}  1&2&3\\1&1&3
\end{pmatrix},
\begin{pmatrix}  1&2&3\\2&2&3
\end{pmatrix},\\
\vspace{0.1cm}
\hspace{3.1cm}
\begin{pmatrix}  1&2&3\\1&1&1
\end{pmatrix},
\begin{pmatrix}  1&2&3\\2&2&2
\end{pmatrix},
\begin{pmatrix}  1&2&3\\3&3&3
\end{pmatrix}
\Bigg\}.$
\end{example}


\section{Dominion of $C_n$ in $T_n$}
\vspace{0.5cm}
Let $E(C_n)$ be the set of idempotents in $C_n$. Now we study the dominion of $C_n$ in $T_n$ by characterize $\alpha^{\prime}$ in $T_n \setminus C_n$ so that $\alpha^{\prime}\alpha\alpha^{\prime}$ dominated by $C_n$.
\vspace{0.2cm}
\begin{proposition}\label{coni}
Let $\alpha$ $\in C_n$ and $\alpha^{\prime}$ $\in T_n \setminus C_n$ defined by\\
 \begin{equation*}
 \tag{3}
    x\alpha^{\prime}= \begin{cases}
    \text{min} (x\alpha^{-1}) & \text{if } x \in Im(\alpha)\\
    s\in[(z\alpha)\alpha^{-1}] \,\,\,  \text{such that} \,\, z\alpha=\text{max} \{m\in Im(\alpha) : m<x\}  & \text{if } x \notin Im(\alpha)
    \end{cases}
\end{equation*}\\
Then \\ 
1- $\alpha\alpha^{\prime}\alpha=\alpha$\\
2- $\alpha \alpha^{\prime} \in E(C_n)$\\
3- $\alpha^{\prime}\alpha \in E(C_n)$\\
4- $\alpha^{\prime} \alpha \alpha^{\prime} \in Dom_{T_{n}}(C_n)$
\end{proposition}  

\begin{proof}
Let $x \in X_n$. First, we will show that $x\alpha\alpha^{\prime}\alpha=x\alpha$. Since $x\alpha$ $\in Im(\alpha)$ and $min((x\alpha)\alpha^{-1}) \in (x\alpha)\alpha^{-1}$, then
$x\alpha \alpha^{\prime}\alpha =(min((x\alpha)\alpha^{-1}))\alpha=x\alpha,$ so $\alpha \alpha^{\prime}\alpha=\alpha$.

\vspace{0.2cm}
For the second part, since $x \in (x\alpha)\alpha^{-1}$, then 
  $x\alpha \alpha^{\prime}=min((x\alpha)\alpha^{-1}) \leq x,$ so $\alpha \alpha^{\prime} \in D_n$. Let $x,y \in X_n$ and $x \leq y$. Since $\alpha \in O_n$, then $max((x\alpha)\alpha^{-1}) \leq min((y\alpha)\alpha^{-1})$,
  and $x\alpha \alpha^{\prime} =min((x\alpha)\alpha^{-1}) \leq min((y\alpha)\alpha^{-1})=y\alpha \alpha^{\prime},$
  so $\alpha\alpha^{\prime}\in O_n$, thus $\alpha\alpha^{\prime}\in O_n \cap D_n=C_n$.
  Also
$(\alpha\alpha^{\prime})^2=\alpha\alpha^{\prime}$, so $\alpha \alpha^{\prime} \in E(C_n)$.

\vspace{0.1cm}
Next, for the third part,
  if $x \in Im(\alpha)$, then $x \alpha^{\prime}\alpha=(min(x\alpha^{-1}))\alpha=x.$
  If $x \notin Im(\alpha)$, then 
  $x \alpha^{\prime}\alpha  =s\alpha=z\alpha<x$,
  thus $\alpha^{\prime}\alpha \in D_n$. Now let $x,y \in X_n$ and $x \leq y$, we consider four cases. Let
  $$ z_1\alpha=max\{m \in Im(\alpha): m<x\} \ \ \ \text{and} \ \ \ z_2\alpha=max\{m \in Im(\alpha): m<y\}.$$
  If $x,y \in Im(\alpha)$, then
  $x \alpha^{\prime}\alpha=x \leq y=y\alpha^{\prime}\alpha$.\\
  If $x \notin Im(\alpha)$ but $y \in Im(\alpha)$, then
  $x \alpha^{\prime}\alpha=z_1\alpha<x \leq y=y\alpha^{\prime}\alpha$.\\
  If $x \in Im(\alpha)$ but $y \notin Im(\alpha)$, then
  $y \alpha^{\prime}\alpha=z_2\alpha$. Since $x \leq y$ and $z_2\alpha<y$ but $z_2\alpha$ is the maximum, then $x \alpha^{\prime}\alpha=x<z_2\alpha=y \alpha^{\prime}\alpha$.\\
  If $x,y \notin Im(\alpha)$, then
$x \alpha^{\prime}\alpha=z_1\alpha$ and $y \alpha^{\prime}\alpha=z_2\alpha$. Since $z_1\alpha<x \leq y$ and $z_2\alpha<y$, but $z_2\alpha$ is the maximum, then $x \alpha^{\prime}\alpha=z_1\alpha<z_2\alpha=y \alpha^{\prime}\alpha$, thus $\alpha^{\prime}\alpha \in O_n$ and so $\alpha^{\prime}\alpha \in O_n \cap D_n=C_n$. Also
$(\alpha^{\prime}\alpha)^2=\alpha^{\prime}\alpha$, so $\alpha^{\prime}\alpha \in E(C_n)$.

\vspace{0.1cm}
For the last part, let $f,g : T_n \rightarrow T$ be semigroup homomorphisms with $f|_{C_{n}}=g|_{C_{n}}$. \\
Then
\ \ \ \ \ \ 
$(\alpha^{\prime}\alpha\alpha^{\prime})f
= (\alpha^{\prime}\alpha)f(\alpha^{\prime})f\\
\vspace{0.1cm}
\hspace{86pt}
=(\alpha^{\prime}\alpha)g(\alpha^{\prime})f \hspace{50pt}[\alpha^{\prime}\alpha \in C_n]\\
\vspace{0.1cm}
\hspace{86pt}= (\alpha^{\prime})g(\alpha)g(\alpha^{\prime})f\\
\vspace{0.1cm}
\hspace{86pt}
= (\alpha^{\prime})g(\alpha)f(\alpha^{\prime})f \hspace{35pt}[\alpha \in C_n]
\\
\vspace{0.1cm}
\hspace{86pt}= (\alpha^{\prime})g(\alpha\alpha^{\prime})f\\
\vspace{0.1cm}
\hspace{86pt}
= (\alpha^{\prime})g(\alpha\alpha^{\prime})g \hspace{50pt}[\alpha \alpha^{\prime} \in C_n]
\\
\vspace{0.1cm}
\hspace{86pt}= (\alpha^{\prime}\alpha\alpha^{\prime})g$.\\
Hence, $\alpha^{\prime} \alpha \alpha^{\prime} \in Dom_{T_n}(C_n)$.\\
\end{proof}

Next we characterize $\alpha \in C_n$ so that $\alpha^{\prime}\alpha\alpha^{\prime} \in Dom_{T_n}(C_n) \setminus C_n$.
\vspace{0.1cm}
\begin{lemma} \label{EcN}
Let $\alpha$ $\in C_n$ and $\alpha^{\prime}$ $\in T_n \setminus C
_n$ defined by (3). \\
1- If $\alpha \in E(C_n)$, then $\alpha^{\prime}\alpha\alpha^{\prime} \in C_n$.\\
2- If $\alpha \notin E(C_n)$, then $\alpha^{\prime}\alpha\alpha^{\prime} \in Dom_{T_n}(C_n) \setminus C_n$.
\end{lemma}

\begin{proof}
   1- Let $\alpha \in E(C_n)$ and $x \in X_n$. We will show that $\alpha^{\prime}\alpha\alpha^{\prime} \in O_n \cap D_n$. If $x \in Im(\alpha)$, then 
   $$x \alpha^{\prime}\alpha\alpha^{\prime} =(min(x\alpha^{-1}))\alpha\alpha^{\prime}=x\alpha^{\prime}=min(x\alpha^{-1})=x.$$
   If $x \notin Im(\alpha)$, then 
   $$x \alpha^{\prime}\alpha\alpha^{\prime} =s\alpha\alpha^{\prime}=z\alpha\alpha^{\prime}=(min((z\alpha)\alpha^{-1}))=z\alpha < x,$$ thus for all $x \in X_n$, $x\alpha^{\prime}\alpha\alpha^{\prime} \leq x$, so $\alpha^{\prime}\alpha\alpha^{\prime}\in D_n$.

   \vspace{0.2cm}
   From Proposition \ref{coni}, we observe that if $\alpha \in E(C_n)$, then  $$x\alpha^{\prime}\alpha\alpha^{\prime}=x=x\alpha^{\prime}\alpha \ \ (\text{if} \ x \in Im(\alpha)),\ \ \ \text{and} \ \ \ x\alpha^{\prime}\alpha\alpha^{\prime}=z\alpha=x\alpha^{\prime}\alpha \ \ (\text{if} \ x \notin Im(\alpha)),$$ 
   thus for all $x,y \in X_n$ and $x \leq y$, $x\alpha^{\prime}\alpha\alpha^{\prime} \leq y \alpha^{\prime}\alpha\alpha^{\prime}$, so $\alpha^{\prime}\alpha\alpha^{\prime} \in O_n$. Hence $\alpha^{\prime}\alpha\alpha^{\prime} \in D_n \cap O_n=C_n$. 

\vspace{0.3cm}
  2- Let $\alpha \notin E(C_n)$. Then $\exists \ x \in Im(\alpha)$ such that $x \notin x\alpha^{-1}$,
   so 
   $$x \alpha^{\prime}\alpha\alpha^{\prime} =(min(x\alpha^{-1}))\alpha\alpha^{\prime}=x\alpha^{\prime}=min (x\alpha^{-1}).$$ Since $\forall \ y \in x\alpha^{-1}, y >x$, then $x\alpha^{\prime}\alpha\alpha^{\prime} > x$,
  and so 
$\alpha^{\prime}\alpha\alpha^{\prime} \notin D_n$, thus
  $\alpha^{\prime}\alpha\alpha^{\prime} \in Dom_{T_n}(C_n) \setminus C_n$.\\
  \end{proof}

\vspace{0.1cm}
  \begin{remark}
It is not guaranteed that $\alpha^{\prime}\alpha\alpha^{\prime}=\alpha^{\prime}$. For example,\\

\vspace{0.3cm}
let \
$\alpha=
\begin{pmatrix}  1&2&3&4&5\\1&1&2&4&4
\end{pmatrix}
$ $\in C_5,$ then
$\alpha^{\prime}=
\begin{pmatrix}  1&2&3&4&5\\1&3&3&4&5
\end{pmatrix}$ $\in T_5$,
so  
$\alpha^{\prime}\alpha\alpha^{\prime}=
\begin{pmatrix}  1&2&3&4&5\\1&3&3&4&4
\end{pmatrix}
$$\neq \alpha^{\prime}$.
\end{remark}

\vspace{0.8cm}
In the following result, we will see that for which $\alpha^{\prime}$ this equality will be true.
\vspace{0.1cm}
\begin{lemma}
Let $\alpha$ $\in C_n$ and $\alpha^{\prime}$ $\in T_n \setminus C_n$ defined by
 \begin{equation}
 \tag{4}
    x\alpha^{\prime}=
    \begin{cases}
    \text{min} (x\alpha^{-1}) & \text{if } x \in Im(\alpha)\\
    min((z\alpha)\alpha^{-1}) \ \ \text{such that} \,\, \, z\alpha=\text{max} \{m\in Im(\alpha) : m<x\}  & \text{if } x \notin Im(\alpha)
    \end{cases}
\end{equation}\\
Then $\alpha^{\prime}\alpha\alpha^{\prime}=\alpha^{\prime}$.
 \end{lemma}

\begin{proof}
    Let $x \in X_n$. We will show that $x\alpha^{\prime}\alpha\alpha^{\prime}=x\alpha^{\prime}$. If $x \in Im(\alpha)$, then $$x\alpha^{\prime}\alpha\alpha^{\prime} =(min(x\alpha^{-1}))\alpha\alpha^{\prime}=x\alpha^{\prime}.$$
    If $x \notin Im(\alpha)$, then   $$x\alpha^{\prime}\alpha\alpha^{\prime} =(min((z\alpha)\alpha^{-1}))\alpha\alpha^{\prime}= z\alpha\alpha^{\prime}=(min((z\alpha)\alpha^{-1}))=x\alpha^{\prime},$$
    thus $\alpha^{\prime}\alpha\alpha^{\prime}=\alpha^{\prime}$.\\
\end{proof}


In what follows we show that it is only one element $\alpha^{\prime\prime} \in T_n$ that we need to add to $C_n$ to generate all elements of $Dom_{T_{n}}(C_n)$. Moreover, it will be shown that $Dom_{T_{n}}(C_n)$ is the smallest regular semigroup containing $C_n$.

\vspace{0.2cm}
\begin{proposition} \label{pcn}
Let $\alpha \in C_n$ defined by 

\vspace{0.5cm}
\hspace{6cm} $x\alpha=max(1,x-1), \ \ \forall \ x\in X_n.$\\[10pt]
and $\alpha^{\prime\prime}$ is given by\\
\vspace{0.1cm}
\hspace{6cm} $x\alpha^{\prime\prime}=min(x+1,n), \ \ \forall \ x>1, \ 1\alpha^{\prime\prime}=1$.\\[10pt]
Then $Dom_{T_n}(C_n)= \ <C_n \bigcup \{{\alpha^{\prime\prime}}\}>$.
\end{proposition}

\vspace{0.8cm}
Before proving this proposition, we need to prove several lemmas as follows.

\vspace{0.8cm}
Consider the semigroup \
$O_{n}^{(1)}=\{\alpha \in O_n : 1\alpha=1\},$
let
$J^*_{r}=\{\alpha \in O_{n}^{(1)} : |Im(\alpha)|=r\},$
where $1 \leq r \leq n-1$, and let $E_{n-1}$ the set of idempotents in $J^*_{n-1}$.

\vspace{0.8cm}
We will adopt the proof of Lemma 6.3.2 in \cite{6} to prove the following lemma
\vspace{0.1cm}
\begin{lemma}
 $J^*_{r}\subseteq \ <J^*_{n-1}>=O_{n}^{(1)}.$ 
\end{lemma}

\begin{proof}
Let

$$\alpha=
\begin{pmatrix}  A_1&A_2&A_3...&A_r\\1&b_2&b_3...&b_r
\end{pmatrix}
\in J^*_r.$$\\[5pt]
and let $b_i\alpha^{-1}=A_i \,\,\, (i=2,3,...,r)$, $ 1\alpha^{-1}=A_1$ and $1\in A_1$. Since not all of the sets $A_i$ are singletons, we may assume without loss of generality that $A_k=\{s,t\}$, where $t=max(A_k)$. For all $x\in X_n$, we define $\epsilon: X_n \rightarrow X_n$ by

$$t\epsilon=t-1, \ \ \ \ x\epsilon=x \ (x \neq t),$$ 
and define $\beta: X_n \rightarrow X_n$ by
$$t\beta=y \ \ (y \notin Im(\alpha) \ \text{where} \ b_k<y<b_{k+1}), \ \ \ \ \ x\beta=b_k \ \ (x \in A_k\setminus \{t\}), \ \ \ \ \ x\beta=b_i \ \ (x \in A_i, \ i<k \ \text{or} \ i>k).$$
Then $\epsilon \in E_{n-1}$ and $|Im(\beta)|=|Im(\alpha)|+1$, so $\beta \in J^*_{r+1}$. Now we have that

 $$t\epsilon\beta=b_k, \ \ \ \ \ x\epsilon\beta=b_k \ \ (x \in A_k\setminus \{t\}), \ \ \ \ \ x\epsilon\beta=b_i \ \ (x \in A_i, \ i<k \ \text{or} \ i>k).$$

Thus for all $x\in X_n$, $x\epsilon\beta=x\alpha$, and so $\alpha=\epsilon\beta$.\\
\end{proof}

As a consequence of this lemma, we deduce that $O^{(1)}_n=<J^*_{n-1}>$.

\begin{lemma}
  $J^*_{n-1}\subseteq \ <C_n \bigcup \{\alpha^{\prime\prime}\}>$.
\end{lemma}

\begin{proof}
Let $\alpha$ and $\alpha^{\prime\prime}$ defined by Proposition \ref{pcn},
$$\alpha=
\begin{pmatrix}  
1&2&3&4&...&n-1&n\\1&1&2&3&...&n-2&n-1
\end{pmatrix} \ \ \ \text{and} \ \ \ 
\alpha^{\prime\prime}=
\begin{pmatrix}  
1&2&3&4&...&n-1&n\\1&3&4&5&...&n&n
\end{pmatrix}$$
and let $\alpha^{\prime}$ be a bijection from $\{2,3,...,n-1\}$ to $\{3,4,...,n\}$, also $min(1\alpha^{-1})=1$, $min(x\alpha^{-1})=x+1 \ (2\leq x <n$), and \ $n\alpha^{\prime}=n$, so
$\alpha^{\prime}=\alpha^{\prime\prime}.$
Now let $f_{u,u-1}$ be the decreasing idempotent maps in $J^*_{n-1}$, then
$$f_{u,u-1}=
\begin{pmatrix}  
1&2&...&u-1&u&u+1&...&n\\1&2&...&u-1&u-1&u+1&...&n
\end{pmatrix}, \ \ (1\leq u-1 <u \leq n).$$
We see that
$f_{u,u-1} \in \ <C_n \bigcup \{\alpha^{\prime\prime}\}>.$
Let $g_{r,r+1}$ be the increasing idempotent maps in $J^*_{n-1}$, then
$$g_{r,r+1}=
\begin{pmatrix}  
1&2&...&r&r+1&r+2&...&n\\1&2&...&r+1&r+1&r+2&...&n
\end{pmatrix}, \ \ (1<r<r+1 \leq n).$$
We want to generate $g_{r,r+1}$ by some $f_{u,u-1}$ and
$\alpha^{\prime\prime}$,
and we will do that by considering two cases.

\vspace{0.2cm}
Case (1) If $r=2$, then
$$f_{3,2} f_{n,3} \alpha^{\prime\prime}=
\begin{pmatrix}  
1&2&3&4&...&n-1&n\\1&3&3&5&...&n&4
\end{pmatrix}$$
have cycles of lengths $1$, $1$ and $n-3$, $( \{1,1\}, \{3,3\},$ and $\{4,5,...,n-1,n\},$ resp.). Since $2$ linked to the cycle $\{3,3\}$ of length $1$, then $f_{3,2} f_{n,3} \alpha^{\prime\prime}(2)=
(f_{3,2} f_{n,3} \alpha^{\prime\prime})^{lcm(1,n-3)}(2)=3$, and so $(f_{3,2} f_{n,3} \alpha^{\prime\prime})^{(n-3)}=g_{2,3}.$

\vspace{0.2cm}
Case (2) If $r>2$, then $g_{r,r+1}$ has two cases depending on $r$.

\vspace{0.2cm}
For $2<r<n-1$, then
$$f_{r+1,r}f_{n,r+1}\alpha^{\prime\prime}f_{r,2}=
\setcounter{MaxMatrixCols}{14}
\begin{pmatrix}
1&2&...&r-1&r&r+1&r+2&...&n-1&n\\1&3&...&2&r+1&r+1&r+3&...&n&r+2
\end{pmatrix}$$
have cycles of lengths $1$, $r-2$, $1$ and $n-r-1$, $( \{1,1\}, \{2,3,...,r-1\}, \{r+1,r+1\}$ and $\{r+2,r+3...,n-1,n\}$, resp.). Since $r$ linked
to the cycle $\{r+1,r+1\}$ of length $1$, then $f_{r+1,r}f_{n,r+1}\alpha^{\prime\prime}f_{r,2}(r)=(f_{r+1,r}f_{n,r+1}\alpha^{\prime\prime}f_{r,2})^{lcm(r-2,n-r-1)}(r)=r+1,$ and so $(f_{r+1,r}f_{n,r+1}\alpha^{\prime\prime}f_{r,2})^{lcm(r-2,n-r-1)}=g_{r,r+1}.$

\vspace{0.2cm}
For $r=n-1$, then
$$f_{n,n-1} \alpha^{\prime\prime}f_{n-1,2}=
\begin{pmatrix}  
1&2&3&...&n-2&n-1&n\\1&3&4&...&2&n&n
\end{pmatrix}$$
have cycles of lengths $1$, $n-3$, and $1$,  $( \{1,1\}, \{2,3,4,...,n-2\},$ and $\{n,n\}$, resp.). Since $n-1$ linked to the cycle $\{n,n\}$ of length $1$, then $f_{n,n-1} \alpha^{\prime\prime}f_{n-1,2}(n-1)=(f_{n,n-1} \alpha^{\prime\prime}f_{n-1,2})^{lcm(1,n-3)}(n-1)=n$, and so $(f_{n,n-1} \alpha^{\prime\prime}f_{n-1,2})^{(n-3)}=g_{n-1,n}$.

From case (1) and case (2), we conclude that 
\begin{equation*}
    g_{r,r+1}= \begin{cases}
     (f_{3,2}f_{n,3}\alpha^{\prime\prime})^{t} & \text{if } r=2\\
    (f_{r+1,r}f_{n,r+1}\alpha^{\prime\prime}f_{r,2})^{t} & \text{if } r>2
    \end{cases}
\end{equation*}
where $t=lcm(r-2,n-(r+1))$ and $r \neq n$. Therefore, $J^*_{n-1}\subseteq \ <C_n \bigcup \{\alpha^{\prime\prime}\}>$.
\end{proof}

\begin{lemma}
  Let $\beta \in T_n$ be such that $1\beta>1$. Then $\beta \notin <C_n \bigcup \{{\alpha^{\prime\prime}}\}>$.
\end{lemma}
\begin{proof}
   Let $\beta \in T_n$. Since $\forall \ \alpha\in C_n$, $1\alpha=1$ and $1\alpha^{\prime\prime}=1$, then $\beta \notin <C_n \bigcup \{{\alpha^{\prime\prime}}\}>$.\\
\end{proof}

\begin{corollary}\label{ccn}
  $<C_n \bigcup \{{\alpha^{\prime\prime}}\}> \ =O_{n}^{(1)}$.  
\end{corollary}

\vspace{0.5cm}
Now to prove Proposition \ref{pcn}, we observe that every elements of $Dom_{T_n}(C_n)$ belong to $O_{n}^{(1)}$, then by Corollary \ref{ccn} we have that 
$$Dom_{T_n}(C_n) \subseteq \ <C_n \bigcup \{{\alpha^{\prime\prime}}\}>.$$

\vspace{0.5cm}
For the second inclusion, we want to prove that $\alpha\alpha^{\prime\prime}\alpha, \ \alpha\alpha^{\prime\prime},\ \alpha^{\prime\prime}\alpha$ and $\alpha^{\prime\prime}$ in $Dom_{T_n}(C_n)$. 
In the next four cases, we assume that $x,y \in X_n$ and $x\leq y$.

\vspace{0.3cm}
 For $\alpha\alpha^{\prime\prime}$ and $\alpha\alpha^{\prime\prime}\alpha$.
 
If $x=1$ and $y=2$, then $1\alpha\alpha^{\prime\prime}=1\alpha^{\prime\prime}=1$ and $2\alpha\alpha^{\prime\prime}=1\alpha^{\prime\prime}=1$.\\
If $x=1$ and $y>2$, then $1\alpha\alpha^{\prime\prime}=1$ and $y\alpha\alpha^{\prime\prime}=(y-1)\alpha^{\prime\prime}=y$.\\
If $x=2$ and $y>2$, then $2\alpha\alpha^{\prime\prime}=1$ and $y\alpha\alpha^{\prime\prime}=y$.\\
If $x>2$ and $y>2$, then $x\alpha\alpha^{\prime\prime}=x$ and $y\alpha\alpha^{\prime\prime}=y$.\\
Then $x\alpha\alpha^{\prime\prime} \leq y\alpha\alpha^{\prime\prime}$.\\
Also,\\
If $x=1$ and $y=2$, then $1\alpha\alpha^{\prime\prime}\alpha=1\alpha=1$ and $2\alpha\alpha^{\prime\prime}\alpha=1\alpha=1$.\\
If $x=1$ and $y>2$, then $1\alpha\alpha^{\prime\prime}\alpha=1$ and $y\alpha\alpha^{\prime\prime}\alpha=y\alpha=y-1$.\\
If $x=2$ and $y>2$, then $2\alpha\alpha^{\prime\prime}\alpha=1$ and $y\alpha\alpha^{\prime\prime}\alpha=y-1$.\\
If $x>2$ and $y>2$, then $x\alpha\alpha^{\prime\prime}\alpha=x-1$ and $y\alpha\alpha^{\prime\prime}\alpha=y-1$.\\
 Then $x\alpha\alpha^{\prime\prime}\alpha \leq y\alpha\alpha^{\prime\prime}\alpha$, so by Proposition \ref{ndn}, $\alpha\alpha^{\prime\prime}, \alpha\alpha^{\prime\prime}\alpha \in O_n \cap D_n=C_n \subseteq Dom_{T_{n}}(C_n)$.

\vspace{0.3cm}
For $\alpha^{\prime\prime}\alpha$ and $\alpha^{\prime\prime}$.\\
If $x=1$ and $y=n$, then $1\alpha^{\prime\prime}\alpha=1\alpha=1$ and $n\alpha^{\prime\prime}\alpha=n\alpha=n-1$.\\
If $x=1$ and $1<y<n$, then $1\alpha^{\prime\prime}\alpha=1$ and $y\alpha^{\prime\prime}\alpha=(y+1)\alpha=y$.\\
If $1<x<n$ and $y=n$, then $x\alpha^{\prime\prime}\alpha=x$ and $n\alpha^{\prime\prime}\alpha=n-1$.\\
If $1<x<n$ and $1<y<n$, then $x\alpha^{\prime\prime}\alpha=x$ and $y\alpha^{\prime\prime}\alpha=y$.\\
Then $x\alpha^{\prime\prime}\alpha \leq y\alpha^{\prime\prime}\alpha$, so by Proposition \ref{ndn}, $\alpha^{\prime\prime}\alpha \in O_n \cap D_n= C_n \subseteq Dom_{T_{n}}(C_n)$.\\
Also,\\
if $x=1$, then $1\alpha^{\prime\prime}\alpha\alpha^{\prime\prime}=1\alpha^{\prime\prime}$.\\
if $x=n$, then 
$n\alpha^{\prime\prime}\alpha\alpha^{\prime\prime}=(n-1)\alpha^{\prime\prime}=n=n\alpha^{\prime\prime}$.\\
if $1<x<n$, then 
$x\alpha^{\prime\prime}\alpha\alpha^{\prime\prime}=x\alpha^{\prime\prime}$.\\
Then $\alpha^{\prime\prime}=\alpha^{\prime\prime}\alpha\alpha^{\prime\prime}$, so by Proposition \ref{coni}, $\alpha^{\prime\prime}\in Dom_{T_{n}}(C_n)$. Thus 

$$<C_n \bigcup \{{\alpha^{\prime\prime}}\}> \ \subseteq Dom_{T_{n}}(C_n).$$\\[10pt]
From the above, we deduce that $Dom_{T_n}(C_n)= \ <C_n \bigcup \{{\alpha^{\prime\prime}}\}> \ =O_{n}^{(1)}$.

Now we want to prove that $Dom_{T_{n}}(C_n)$ is regular. Let $\alpha \in O^{(1)}_n$ and define a mapping $\beta:X_n \rightarrow X_n$ as follows:  
    \begin{equation*}
    x\beta= \begin{cases}
     min(x\alpha^{-1}) & \text{if} \ \ x\in Im(\alpha) \\
    min((z\alpha)\alpha^{-1}) \ \text{such that} \ z\alpha=max\{m \in Im(\alpha): m < x\}\ & \text{if} \ \ x\notin Im(\alpha)
    \end{cases}
\end{equation*}
Then for all $x \in X_n$,
$$x\alpha\beta\alpha=min((x\alpha)\alpha^{-1})\alpha=x\alpha,$$
it follows that $\alpha\beta\alpha=\alpha$. Let $x,y \in X_n$ and $x \leq y$, we consider four cases.  Let
  $$ z_1\alpha=max\{m \in Im(\alpha): m<x\} \ \ \ \text{and} \ \ \ z_2\alpha=max\{m \in Im(\alpha): m<y\}.$$

Case (1) If $x, y \in Im(\alpha).$
Since $\alpha \in O_n$, then $max(x\alpha^{-1}) \leq min(y\alpha^{-1})$, so $$x\beta=min(x\alpha^{-1}) \leq min(y\alpha^{-1})=y \beta.$$

\vspace{0.1cm}
Case (2) If $x \notin Im(\alpha)$ but $y \in Im(\alpha)$, then $x\beta=min((z_1\alpha)\alpha^{-1})$ and $y\beta=min(y\alpha^{-1})$. Since $z_1\alpha < x \leq y$, then $z_1\alpha < y$, so $$x\beta=min((z_1\alpha)\alpha^{-1})<min(y\alpha^{-1})=y\beta.$$

\vspace{0.1cm}
Case (3) If $x \in Im(\alpha)$ but $y \notin Im(\alpha)$, then
$x\beta=min(x\alpha^{-1})$ and $y\beta=min((z_2\alpha)\alpha^{-1})$. Since $x \leq y$ and $z_2\alpha < y$ but $z_2\alpha$ is the maximum, then $$x\beta=min(x\alpha^{-1}) \leq min((z_2\alpha)\alpha^{-1})=y\beta.$$

\vspace{0.1cm}
Case (4) If $x , y \notin Im(\alpha)$, then $x\beta=min((z_1\alpha)\alpha^{-1})$ and $y\beta=min((z_2\alpha)\alpha^{-1})$. Since $z_1\alpha < x \leq y$ and $z_2\alpha < y$, but $z_2\alpha$ is the maximum, then $$x\beta=min((z_1\alpha)\alpha^{-1}) \leq min((z_2\alpha)\alpha^{-1})=y\beta.$$
Thus $x\beta \leq y\beta$, so $\beta \in O_n$. Since $1 \in Im(\alpha)$, then $1 \beta=min(1\alpha^{-1})=1$, so $\beta \in O_n^{(1)}$. Hence $O_n^{(1)}$ is regular.

\vspace{0.3cm}
\begin{corollary}
   $Dom_{T_{n}}(C_n)=Dom_{T_{n}}(O_n) \cap Dom_{T_{n}}(D_n).$
\end{corollary}

\vspace{0.5cm}
\begin{lemma}  
$Dom_{T_n}(C_n)$ is idempotent-generated.
\end{lemma}

\begin{proof}
  As we known that $C_n$ is idempotent-generated \cite{HI},
and 
$\alpha^{\prime\prime}=
\begin{pmatrix}  
1&2&3&4&...&n-1&n\\1&3&4&5&...&n&n
\end{pmatrix} \notin C_n.$
 We need to show that $\alpha^{\prime\prime}$ is a product of idempotents in $O_n^{(1)}.$ For all $x \in X_n$ we define $\epsilon_{i}: X_n \rightarrow X_n$, where $i=2,3,4,...,n-1$, by
$$ i\epsilon_{i}=i+1, \ \ \ \ \ \ \ \ \ x\epsilon_{i}=x \ \ (x \neq i).$$
Then $\epsilon_{i}$ are idempotents in $O_n^{(1)}$. Now let $\beta=\epsilon_{n-1}\epsilon_{n-2}...\epsilon_{3}\epsilon_{2}$, then
$$1\beta=1, \ \ \ \ \ \ \ \ \  \ \ \ x\beta=x+1 \ \ (2 \leq x \leq n-1), \ \ \ \ \ \ \ \ \ \ \ \ n\beta=n,$$
it is not difficult to see that 
$$\alpha^{\prime\prime}=\epsilon_{n-1}\epsilon_{n-2}...\epsilon_{3}\epsilon_{2},$$
and so $Dom_{T_n}(C_n)=<C_n \bigcup \{\ \alpha^{\prime\prime} \}>$ is idempotent-generated.\\
\end{proof}

\vspace{0.3cm}
For natural numbers $k, m, n, p$ and $r$ we have

\vspace{0.3cm}
\begin{result} \cite{ria} \label{one}
    $\sum_{k=0}^{n}\binom{n}{m-k}\binom{p}{k}=\binom{n+p}{m}.$
\end{result}

\begin{result} \cite{ria} \label{three}
  $\sum_{m=r}^{n}\binom{m}{r}\binom{n+k-m}{k}=\binom{n+k+1}{r+k+1}.$
\end{result}

\begin{result} \cite{lupt} \label{five}
   $\sum_{k=r}^{n}\binom{k-1}{r-1}=\binom{n}{r}.$ 
\end{result}

\vspace{0.5cm}
Now we give a formula to find the number of elements and idempotents of dominion $C_n$ in $T_n$. The next results are similar to the results in \cite{nrk} with some essential differences. For $1 \leq r \leq k \leq n$, we define

\begin{equation}
 J^{(1)}(n,r,k)=|\{\alpha \in O_n^{(1)} : |Im(\alpha)|=r \ \wedge \ max(Im(\alpha))=k\}|.   \label{lh9}
\end{equation}
Then we have

\vspace{0.2cm}
\begin{proposition}
    Let $J^{(1)}(n,r,k)$ be ss defined in \eqref{lh9}. Then
    $$J^{(1)}(n,r,k)=\binom{k-2}{r-2}\binom{n-1}{r-1}.$$
\end{proposition}
\begin{proof}
If $r=1$, then $k=1$ and so $J^{(1)}(n,1,1)=1$. Now for $r\geq 2$, since $|Im(\alpha)|=r$, we have to partition $X_n$ into $r$ convex classes, however, this can be done in $\binom{n-1}{r-1}$ ways, by inserting $r-1$ symbols between the $n-1$ spaces of $X_n$. Moreover, since $max(Im(\alpha))=k$, therefore we can choose the remaining $r-2$ elements of $Im(\alpha)$ from $\{2,3,...,k-1\}$ in $\binom{k-2}{r-2}$ ways, and there is only one way of tying them to the $r$ convex classes in an order-preserving fashion.\\
\end{proof}

\begin{corollary}
 Let $J^{(1)}(n,r)=\sum_{k=r}^{n}J^{(1)}(n,r,k)$ . Then $J^{(1)}(n,r)=\binom{n-1}{r-1}\binom{n-1}{r-1}.$
\end{corollary}
\begin{proof}
$J^{(1)}(n,r)=\sum_{k=r}^{n} \ J^{(1)}(n,r,k)=\binom{n-1}{r-1}\sum_{k=r}^{n} \ \binom{k-2}{r-2}.$\\[10pt]
Put $j=k-1$. If $k=r$, then $j=r-1$, if $k=n$, then $j=n-1$, so by Result \ref{five}

$$J^{(1)}(n,r)=\binom{n-1}{r-1}\sum_{j=r-1}^{n-1} \ \binom{j-1}{(r-1)-1}=\binom{n-1}{r-1}\binom{n-1}{r-1}.$$
\end{proof}

\vspace{0.5cm}
\begin{corollary}
    Let $G^{(1)}(n,k)=\sum_{r=1}^{k}J^{(1)}(n,r,k)$ . Then $G^{(1)}(n,k)=\binom{n+k-3}{k-1}.$
\end{corollary}
\begin{proof}
$G^{(1)}(n,k)=\sum_{r=1}^{k} \ J^{(1)}(n,r,k)=\sum_{r=1}^{k} \ \binom{k-2}{r-2}\binom{n-1}{r-1}.$\\[10pt]
Put $j=r-1$. If $r=1$, then $j=0$, if $r=k$, then $j=k-1$, so

$$G^{(1)}(n,k)=\sum_{j=0}^{k-1} \ \binom{k-2}{j-1}\binom{n-1}{j}
=\sum_{j=1}^{k-1} \ \binom{k-2}{j-1}\binom{n-1}{j}$$
Put $i=j-1$. If $j=1$, then $i=0$, if $j=k-1$, then $i=k-2$, so by Result \ref{one}

$$G^{(1)}(n,k)=\sum_{i=0}^{k-2} \ \binom{k-2}{i}\binom{n-1}{i+1}
=\sum_{i=0}^{k-2} \ \binom{k-2}{i}\binom{n-1}{(n-2)-i}=\binom{(k-2)+(n-1)}{n-2}$$
$$\ \ \ \ \ \ \ \ \ \ \ \ \ \ \ \ \ \ \ \ \ =\binom{n+k-3}{n-2}=\binom{n+k-3}{k-1}.$$
\end{proof}

\begin{table}[h!]
    \centering
     \caption{$J^{(1)}(n,r)$}
    \renewcommand{\arraystretch}{1.2}
    \begin{tabular}{|c|c|c|c|c|c|c|c|c|c|}
    \hline
    & \multicolumn{8}{c}{$ \ \ \ \  \ \ \ \ \ \ \ \ \ r$} & \\
    \hline
    $n$ & 1 & 2 & 3 & 4 & 5 & 6 & 7 & 8 & $\sum \ J^{(1)}(n,r)$\\
    \hline
         1 & 1 & & & & & & & & 1 \\
         \hline
         2 & 1 & 1 & & & & & & & 2 \\
         \hline
         3 & 1 & 4 & 1 & & & & & & 6 \\
         \hline
         4 & 1 & 9 & 9 & 1 & & & & & 20 \\
         \hline
         5 & 1 & 16 & 36 & 16 & 1 & & & & 70 \\
         \hline
         6 & 1 & 25 & 100 & 100 & 25 & 1 & & & 252 \\
         \hline
         7 & 1 & 36 & 225 & 400 & 225 & 36 & 1 & & 924 \\
         \hline
         8 & 1 & 49 & 441 & 1225 & 1225 & 441 & 49 & 1 & 3432 \\
         \hline
    \end{tabular}
    \label{tab:placeholder}\\
\end{table}

\vspace{0.7cm}
\begin{corollary}
   $|O_n^{(1)}|=\sum_{r=1}^{n} J^{(1)}(n,r)=\binom{2n-2}{n-1}.$ 
\end{corollary}
\begin{proof}
   $|O_n^{(1)}|=\sum_{r=1}^{n} J^{(1)}(n,r)=\sum_{r=1}^{n}\binom{n-1}{r-1}\binom{n-1}{r-1}$

\vspace{0.5cm}
   \hspace{4.2cm}
   $=\sum_{r=1}^{n}\binom{n-1}{r-1}\binom{n-1}{(n-1)-(r-1)}$
   
\vspace{0.5cm}
   \hspace{4.2cm}
  $=\sum_{r=1}^{n}\binom{n-1}{r-1}\binom{n-1}{n-r}.$

  \vspace{0.5cm}
Put $k=r-1$. If $r=1$, then $k=0$, if $r=n$, then $k=n-1$, so by Result \ref{one}
$$|O_n^{(1)}|=\sum_{k=0}^{n-1}\binom{n-1}{k}\binom{n-1}{(n-1)-k}=\binom{2n-2}{n-1}$$
\end{proof}

As a consequence of this result, we deduce that

 $$|Dom_{T_n}(C_n)|=\begin{pmatrix} 2n-2 \\ n-1 \end{pmatrix}.$$ 

Now we Define

\begin{equation}
 e^{(1)}(n,r,k)=|\{\alpha \in E(O_n^{(1)}) : |Im(\alpha)|=r \ \wedge \ max(Im(\alpha))=k\}|.   \label{ck3}
\end{equation}

\vspace{0.2cm}
We see that the only idempotent in $O_n^{(1)}$ with image size $1$ is the constant map to $1$, so 
$$e^{(1)}(n,1,1)=1=e^{(1)}(n,r,r)$$
corresponding to $\alpha,\beta \in O_n^{(1)}$ given by
$$x\alpha=1 \ (\forall x\in X_n), \ \ \ x\beta=x \ (1 \leq x \leq r), \ \ \ x\beta=r \ (x>r)$$
respectively. In general. we have

\vspace{0.2cm}
\begin{lemma}\label{e1}
 For natural numbers $2 \leq r \leq k < n$, we have 
 $$e^{(1)}(n,r,k)=e^{(1)}(k,r,k)$$
\end{lemma}
\begin{proof}
  Since $k=max(Im(\alpha))$, it follows that $k\alpha=k$, by idempotency. Moreover,for all $x$  in $\{k+1,k+2,...,n\}$, we have $x\alpha=k$, by order-preservedness. The result is now immediate.
\end{proof}

\vspace{0.5cm}
\begin{lemma} \label{FIG}
  $e^{(1)}(n,r,n)=\sum_{m=1}^{n-r+1} \ m \ e^{(1)}(n-m,r-1,n-m)$
\end{lemma}

\begin{proof}
  If $min (n\alpha^{-1})=\{n-m+1\}\ (1\leq m \leq n-r+1),$ then there are clearly $\sum_{t=r-1}^{n-m} \ e^{(1)}(n-m,r-1,t)$ idempotents. Thus taking the sum over $m$ from $1$ to $n-r+1$ yields
  
  \vspace{0.5cm}
  \hspace{3cm}
  $e^{(1)}(n,r,n)=\sum_{m=1}^{n-r+1}\sum_{t=r-1}^{n-m} \ e^{(1)}(n-m,r-1,t)$
  
  \vspace{0.5cm}
  \hspace{4.8cm}
  $=\sum_{m=1}^{n-r+1}\sum_{t=r-1}^{n-m} \ e^{(1)}(t,r-1,t)\ \ \ \text{(by Lemma \ref{e1}})$
  
   \vspace{0.5cm}
  \hspace{4.8cm}
  $=\sum_{m=1}^{n-r+1} \ m \ e^{(1)}(n-m,r-1,n-m).$
\end{proof}

\vspace{0.5cm}
\begin{proposition}
  Let $e^{(1)}(n,r,k)$ be as defined in \eqref{ck3}. Then for \ $2 \leq r \leq k < n$ 
    $$e^{(1)}(n,r,k)=\binom{k+r-3}{2r-3}.$$
\end{proposition}
\begin{proof}
    First, we show by induction that
    $$e^{(1)}(n,r,n)=\binom{n+r-3}{2r-3}.$$
    If $r=2$, then $Im(\alpha)=\{1,n\}$. Since $1\alpha=1$, every such idempotent has the form
    $$x\alpha=1 \ \ (1 \leq x \leq t),  \ \ \ \ \ \ \ \ \ x\alpha=n \ \ (t+1 \leq x \leq n),$$
    where $1 \leq t \leq n-1$.
    Thus
    $$e^{(1)}(n,2,n)=n-1=\binom{n-1}{1}=\binom{n+2-3}{2(2)-3}.$$\\[10pt]
   So suppose the result is true for all $2 \leq r \leq n$. Now by Result \ref{three} we have
    
    \vspace{0.5cm}
    $e^{(1)}(n+1,r,n+1)=\sum_{m=1}^{n-r+2} \ m \ e^{(1)}(n+1-m,r-1,n+1-m)$
    
  \vspace{0.5cm}
  \hspace{3.2cm}
  $=\sum_{m=1}^{n-r+2} \ m \binom{(n+1-m)+(r-1)-3}{2(r-1)-3}$
  
   \vspace{0.5cm}
  \hspace{3.2cm}
  $=\sum_{m=1}^{n-r+2} \binom{m}{1}\binom{(n-r+2)+(2r-5)-m}{2r-5}$
  
  \vspace{0.5cm}
  \hspace{3.2cm}
  $=\binom{(n-r+2)+(2r-5)+1}{1+(2r-5)+1}$
  
    \vspace{0.5cm}
  \hspace{3.2cm}
  $=\binom{(n+1)+r-3}{2r-3}.$

  \vspace{0.5cm}
  Hence the result for $e^{(1)}(n,r,k)$ is true, by Lemma \ref{e1}.\\
\end{proof}

\vspace{0.5cm}
\begin{corollary}
  Let $e^{(1)}(n,r)=\sum_{k=r}^{n} \ e^{(1)}(n,r,k).$ Then $e^{(1)}(n,r)=\binom{n+r-2}{2r-2}.$
\end{corollary}
\begin{proof}
$e^{(1)}(n,r)=\sum_{k=r}^{n} \ e^{(1)}(n,r,k)=\sum_{k=r}^{n} \ \binom{k+r-3}{2r-3}.$\\[10pt]
Put $j=k+r-2$, if $k=r$, then $j=2r-2$. If $k=n$, then $j=n+r-2$, so by Result \ref{five}

$$e^{(1)}(n,r)=\sum_{j=2r-2}^{n+r-2} \ \binom{j-1}{2r-3}=\binom{n+r-2}{2r-2}.$$
\end{proof}

\vspace{0.5cm}
Define the Fibonacci number $F_m, \ m \geq 1$, as follws
$$F_1=F_2=1 \ \ \ \text{and} \ \ \ F_m=F_{m-1}+F_{m-2} \ \ \ (m \geq 3).$$

\vspace{0.5cm}
Then we have the next result

\begin{corollary}
 Let $g^{(1)}(n,k)=\sum_{r=2}^{k} \ e^{(1)}(n,r,k).$ Then for $k \geq 2$, $g^{(1)}(n,k)=F_{2k-2}.$
 For $k=1$, $g^{(1)}(n,1)=1.$ 
\end{corollary}
\begin{proof}
   $g^{(1)}(n,k)=\sum_{r=2}^{k} \ e^{(1)}(n,r,k)= \sum_{r=2}^{k} \ e^{(1)}(k,r,k)$
   
    \vspace{0.5cm}
  \hspace{5.2cm}
  $=\sum_{r=2}^{k} \binom{k+r-3}{2r-3}$

  \vspace{0.5cm}
  \hspace{5.2cm}
 $=F_{2k-2}.$ 
\end{proof}

\vspace{1cm}
\begin{table}[h!]
    \centering
    \caption{$g^{(1)}(n,k)$}
    \renewcommand{\arraystretch}{1.2}
    \begin{tabular}{|c|c|c|c|c|c|c|c|c|c|}
    \hline
    & \multicolumn{8}{c}{$ \ \ \ \  \ \ \ \ \ \ \ \ \ k$} & \\
    \hline
    $n$ & 1 & 2 & 3 & 4 & 5 & 6 & 7 & 8 & $\sum \ g^{(1)}(n,k)$\\
    \hline
         1 & 1 & & & & & & & & 1 \\
         \hline
         2 & 1 & 1 & & & & & & & 2 \\
         \hline
         3 & 1 & 1 & 3 & & & & & & 5 \\
         \hline
         4 & 1 & 1 & 3 & 8 & & & & & 13 \\
         \hline
         5 & 1 & 1 & 3 & 8 & 21 & & & & 34 \\
         \hline
         6 & 1 & 1 & 3 & 8 & 21 & 55 & & & 89 \\
         \hline
         7 & 1 & 1 & 3 & 8 & 21 & 55 & 144 & & 233 \\
         \hline
         8 & 1 & 1 & 3 & 8 & 21 & 55 & 144 & 377 & 610 \\
         \hline
    \end{tabular}
    \label{tab:placeholder}
\end{table}

\vspace{1cm}
\begin{corollary}
 Let \ $|E(O_n^{(1)})|=1+\sum_{k=2}^{n} \ g^{(1)}(n,k)=F_{2n-1}.$  
\end{corollary}

\begin{proof}
Since for $n \geq 2$, $F_{n+1}=F_{n}+F_{n-1}$, then $F_{n}=F_{n+1}-F_{n-1}$, so 

\vspace{0.9cm}
$|E(O_n^{(1)})|=1+\sum_{k=2}^{n} \ g^{(1)}(n,k)=1+\sum_{k=2}^{n} \ F_{2k-2}$

 \vspace{0.7cm}
 \hspace{4.9cm}
 $=1+\sum_{k=2}^{n} \ (F_{2k-1}-F_{2k-3})$

\vspace{0.7cm}
 \hspace{4.9cm}
 $=1+(F_{3}-F_{1})+(F_{5}-F_{3})+(F_{7}-F_{5})+\dots+(F_{2n-1}-F_{2n-3})$
 
\vspace{0.7cm}
  \hspace{4.9cm}
 $=1+F_{2n-1}-F_{1}$

\vspace{0.7cm}
 \hspace{4.9cm}
 $=1+F_{2n-1}-1$.

\vspace{0.7cm}
\hspace{4.9cm}
 $=F_{2n-1}$.\\
\end{proof}

\vspace{0.5cm}
As a consequence of this result, we deduce that 

 $$|E(Dom_{T_n}(C_n))|=F_{2n-1}.$$


\begin{thebibliography}{99}

\bibitem{7} S.Nasir. A. Umar. \textit{On Dominion of Certain Ample Monoids}. 27, (2023).

\bibitem{cla} O. Ganyushkin and V. Mazorchuk.
\textit{Classical Finite Transformation Semigroups}. An Introduction. Spring. London, (2009).

\bibitem{nrk} A. Laradji A. Umar. \textit{Combinatorial results for semigroups of order-preserving full transformations}. Semigroup Forum, 72, 51-62 (2006).

\bibitem{50} D. S. Dummit. R. M. Foote. \textit{Abstract Algebra}. 3rd ed. Hoboken, NJ: Wiley, (2004).

\bibitem{8}  A.Laradji, A.Umar. \textit{On certain finite semigroups of order-decreasing transformations I}. Semigroup Forum. 69(2), 184-200, (2004).  

\bibitem{lupt}  A.Laradji, A.Umar. \textit{Combinatorial results for semigroups of order-preserving partial transformations}. J. Algebra, 278(1), 342-359 (2004).
       
 \bibitem{6} J. M. Howie. \textit{Fundamentals of semigroup theory}. London Math. Soc. Monographs New-series, (1995).

 \bibitem{HI} P. M. Higgins. \textit{Idempotent depth in semigroups of order-preserving mappings }. Proc. Roy. Soc. Edinburgh. A 124, 1045-1058, (1994).

 \bibitem{HIG} P. M. Higgins. \textit{Combinatorial results for semigroups of order-preserving mappings}. Math. Proc. Camb. Phil. Soc, 113, 281-296 (1993).
 
 \bibitem{15} A.Umar. \textit{On the semigroup of order-decreasing full transformation}. Proc. Roy. Soc. Edinburgh, 129-142, (1992).

  \bibitem{14} P.M.Higgins. \textit{Techniques of semigroup theory}. (Oxford University Press, 1992).

  \bibitem{GM} M.S.Gomes.,J.M.Howie. \textit{On the ranks of certain semigroups of order-preserving transformations}. Semigroup Forum, Vol. 45, 272-282, (1992).


\bibitem{12} P.M.Higgins. \textit{Random products in semigroups of mappings}. In Lattices, Semigroups and Universal Algebra (editors J.Almeida et al), pp. 89-100, (Plenum Press 1990).

\bibitem{16} J. M. Howie., E. F. Robertson, and B. M. Shein. \textit{A combinatorial property of finite full transformation semigroups}. Proc. Roy. Soc. Edinburgh Sect. A 109, 319-328, (1988).


\bibitem{40} J. B. Fountain.  \textit{Abundant semigroups}. Proc. London Math. Soc(3). 44, 103-12, (1982).

\bibitem{196} T. E. Hall, \textit{Epimorphisms and dominions}, Semigroup Forum, 271-283, 24, (1982).

\bibitem{76} J. M. Howie. \textit{An introduction to semigroup theory}. (London: Academic Press, 1976).

\bibitem{30} H. E. Scheiblich. Kayran C. moore. \textit{$T_X$ is absolutely closed}. Semigroup Forum. 6, 216-226, (1973).

\bibitem{pod} J. M. Howie, \textit{Products of idempotents in certain semigroups of transformations}, Proc. Edinburgh Math. Soc. 17, 223-236 (1971).

 \bibitem{ria} J.Riordan. \textit{Combinatorial Identities}. John Wiely and Sons. New Yourk, (1968).

 \bibitem{10} B.Harris. \textit{A note on the number of idempotents in symmetric semigroups}. J.Combin. Theory Ser. A 3, 1234-1235, (1967).

 \bibitem{11} B.Harris. L.Schoenfeld. \textit{The number of idempotent elements in symmetric semigroups}. J.Combin. Theory Ser. A 3, 122-135, (1967).

  \bibitem{13} J. M. Howie and Isbell. \textit{Epimorphisms and dominions II}. Journal of Algebra, 6, 7-21, (1967).

  \bibitem{RI} J. R. Isbell, \textit{Epimorphisms and dominions}, Proceedings of the conference on Categorical Algebra, La Jolla, (1965), 232-246, Lange and Springer, Berlin (1966).

 \bibitem{2} J. M. Howie. \textit{The subsemigroup generated by the idempotents of a full transformation semigroup}. J. London Math. Soc. 41, 707-716, (1966).

\bibitem{AZ} Aizenstat, A.Y. \textit{The defining relations of the endomorphism semigroup of a finite linearly ordered set}, (Russian). Sibirsk Mat.Zh. 3, 161-169, (1962).

\bibitem{1} A. H. Clifford and G. B. Preston. \textit{The algebraic theory of semigrouos}. Vol,1, Math. Surveys of the Amercan Math. Soc., 7, (1961).   

 \bibitem{seq} The On-Line Encyclopedia of Integer Sequences (OEIS) http://oeis.org.     

\end{thebibliography}
\end{document}